\documentclass{article}
\usepackage{amsmath}
\usepackage{amsthm}
\usepackage{amssymb}
\usepackage{bm}
\usepackage[all]{xy}
\usepackage[dvips]{graphicx}
\usepackage{latexsym}
\usepackage{mathtools}
\usepackage{geometry}
 \renewcommand{\equation}

\newtheorem{prop}{Proposition}[section]
\newtheorem{lem}{Lemma}[section]
\newtheorem{thm}{Theorem}[section]

\newtheorem{claim}{Claim}[section]

\theoremstyle{definition}
\newtheorem{rmk}{Remark}[section]
\newtheorem{defini}{Definition}[section]

\title{Generalized torsion elements in the fundamental groups of 3-manifolds obtained by the $0$-surgeries along some double twist knots}
\author{Nozomu Sekino}
\date{}
\begin{document}
\maketitle

\begin{abstract}
We consider the 3-manifold obtained by the $0$-surgery along a double twist knot. 
We construct a candidate for a generalized torsion element in the fundamental group of the surged manifold, and see that there exists the cases where the candidate is actually a generalized torsion element. 
For a proof, we use the JSJ-decomposition of the surged manifold. 
We also prove that the fundamental group of the 3-manifold obtained from the $0$-surgery along a double twist knot is bi-orderable if and only if it admits no generalized torsion elements. 
We also list some examples of the surged manifolds whose fundamental groups admit generalized torsion elements. 
\end{abstract}

\section{Introduction}\label{secintro}
A non-trivial element of a group is called a generalized torsion element if some products of conjugates of it is the identity element. 
The existence of a generalized torsion element prevents the group from being bi-orderable. 
Thus for a group, admitting no generalized torsion elements is a necessary condition for the group being bi-orderable. 
In general, there exists a group which is not bi-orderable and has no generalized torsion elements. See Chapter 4 of \cite{murarhemtulla} and \cite{akhmedovthorne}, for example. 

However, in \cite{motegiteragaito}, it is conjectured that if $F$ is a 3-manifold group i.e the fundamental group of a compact 3-manifold, then non-existence of generalized torsion elements on $F$ implies the bi-orderability of $F$. 
There are many works toward this conjecture \cite{motegiteragaito}, \cite{itomotegiteragaito} \cite{sekino}. 
For example, this conjecture is verified for the fundamental groups of non-hyperbolic geometric 3-manifolds \cite{motegiteragaito}, for those of  3-manifolds obtained by some Dehn surgeries along some knots \cite{itomotegiteragaito} and for those of once punctured torus bundles \cite{sekino}. 

According to this conjecture, sometimes a question whether a given (3-manifold) group admits a generalized torsion element arises. 
In general it is difficult to find a generalized torsion element. 
There are also many works for this question, for some knot groups and link groups \cite{perronrolfsen} \cite{clayrolfsen} \cite{claydesmaraisnaylor} \cite{teragaito1} \cite{teragaito2} \cite{motegiteragaito2}, for the fundamental groups of the 3-manifolds obtained by some Dehn surgeries along some knots \cite{itomotegiteragaito}.  

In this paper, we focus on the fundamental groups of the 3-manifolds obtained by the $0$-surgeries along double twist knots. 
The {\it double twist knot} {\it of type} $(p,q)$, denoted by $K_{p,q}$ in this paper, is the knot represented as $C[2p,2q]$ in the Conway notation. 
For example, $K_{1,1}$ is the figure eight knot and $K_{1,-1}$ is the right-handed trefoil. 
Note that the bridge number of each double twist knot is two. 
%By using the result in \cite{linnellrhemtullarolfsen}, we can see that if a two-bridge knot has an Alexander polynomial whose roots are all real positive, then its knot group is bi-orderable. 
By using a part of the result of \cite{claydesmaraisnaylor}, 
we see that the knot group i.e the fundamental group of the complement of the knot of $K_{p,q}$ is not bi-orderable if $pq$ is negative, and that of $K_{1,q}$ is bi-orderable if $q$ is positive. 
A generalized torsion element of the knot group of $K_{1,q}$ for negative $q$ is constructed \cite{teragaito1}. 
To the author's knowledge, it is unknown whether the knot group of $K_{p,q}$ with positive $pq$ is bi-orderable and whether the knot group of $K_{p,q}$ with negative $pq$ other than that of $K_{1,q}$ admits a generalized torsion element. 

About Dehn surgeries along a double twist knot $K_{p,q}$, the generalized torsion element of the fundamental group $\pi_{1}(K_{p,q}(r))$ of the 3-manifold $K_{p,q}(r)$ obtained by the $r$-surgery along $K_{p,q}$ is constructed in \cite{itomotegiteragaito} with condition on $r$ in terms of $p$ and $q$. 
Since it is known that a finitely generated bi-orderable group must surject onto $\mathbb{Z}$, the only possible surgery along a knot in the 3-sphere producing a 3-manifold whose fundamental group may be bi-orderable is the $0$-surgery. 
About the  $0$-surgeries, it is known that $\pi_{1}(K_{1,1}(0))$ is bi-orderable and $\pi_{1}(K_{1,-1}(0))$ admits a generalized torsion element. 
In this paper, we construct a candidate for a generalized torsion element in other $\pi_{1}(K_{p,q}(0))$, and we see that there exists some $(p,q)$ such that this candidate is actually a generalized torsion element:  
\begin{thm}\label{mainthm}
Let $(p,q)$ and $(p',q')$ be two pairs of non-zero integers such that $pq=p'q'$ and $p\neq \pm p', \pm q'$. 
Then at least one of $\pi_{1} \left( K_{p,q}(0) \right)$ and $\pi_{1}(K_{p',q'}(0))$ admits a generalized torsion element. 
\end{thm}
Note that only by Theorem~\ref{mainthm} we cannot know whether $\pi_{1}(K_{p,q}(0))$ admits a generalized torsion element for a given $(p,q)$. 
We also give computation which judges whether the candidate is a generalized torsion element for some $(p,q)$ in the last section. 
Moreover, we see that there is a generalized torsion element in $\pi_{1}\left(K_{p,q}(0)\right)$ for every double twist knot $K_{p,q}$ with negative $pq$: 
\begin{thm}\label{mainthm2}
Let $(p,q)$ be a pair of non-zero integers such that $pq<0$. 
Then there is a generalized torsion element in $\pi_{1}\left(K_{p,q}(0)\right)$.
\end{thm}
As stated above, the knot group of $K_{p,q}$ with negative $pq$ is not bi-orderable and this implies that there exists a generalized torsion element in the group if the conjecture above is true. 
However, we cannot find a generalized torsion element in the group so far. 
A generalized torsion element constructed at the proof of Theorem~\ref{mainthm2} is obtained through the $0$-surgery. 
Moreover, we have the following, which implies that the conjecture above is true for the fundamental groups of the 3-manifolds obtained by the $0$-surgeries along double twist knots: 
\begin{thm} \label{mainthm3}
Let $(p,q)$ be a pair of non-zero integers. 
Then $\pi_{1}\left( K_{p,q}(0)\right)$ is bi-orderable if and only if it admits no generalized torsion elements. 
\end{thm}

The rest of this paper is constructed as follows. 
In Section~\ref{secgeneralisedtorsion}, we give some definitions and lemmas about generalized torsion elements that are needed later. 
In Section~\ref{secdoubletwist}, we consider the double twist knots, their complements and the 3-manifolds obtained by the $0$-surgeries along these knots. 
We give some presentations of the fundamental groups of such 3-manifolds. 
In Section~\ref{secjsj}, we consider the JSJ-decompositions of the 3-manifolds obtained by the $0$-surgeries along double twist knots. 
Using this, we see that two double twist knots give homeomorphic 3-manifolds by the $0$-surgeries if and only if these knots are equivalent. 
In Section~\ref{secprfofthm}, we give a proof of Theorem~\ref{mainthm} using the results above. A candidate for a generalized torsion element is constructed there. 
In Section~\ref{secprfofthm2}, we give a proof of Theorem~\ref{mainthm2} using some computation. 
%This section is independent of the other sections. 
In Section~\ref{secprfofthm3}, we give a proof of Theorem~\ref{mainthm3}. 
Assuming there are no generalized torsion elements, we give a bi-order. 
In Section~\ref{seccomputations}, we give some results of computation which judges whether the candidate is a generalized torsion element. 

\subsection*{Acknowledgements}
The author would like to thank professor Nozaki for giving him many helpful advices for computations.

\subsection*{Notation}
We give some notations about elements of group. 
The conjugate $h^{-1}gh$ of $g$ with $h$ is represented by the notation $g^{h}$. 
The commutator $g^{-1}h^{-1}gh$ of $g$ and $h$ is represented by the notation $[g,h]$. 
% $\prod \limits_{k=1}^{m}$

\section{Terminologies and Lemmas for generalized torsion elements}\label{secgeneralisedtorsion}
In this section, we define some terminologies and prove lemmas about generalized torsion elements which are needed in later. 
In this section $F$ denotes a group. 

\begin{defini}(Generalized torsion elements)\\
A non-trivial element $g\in F$ is called a {\it generalized torsion element} if there exists a positive integer $n$ and $h_{1},\dots,h_{n}\in F$ such that $g^{h_1}\cdot g^{h_{2}}\cdots g^{h_n}=1_{F}$ holds. 
Minimal such $n$ is called the {\it order} of $g$.
\end{defini} 

%The existing of generalized torsion elements prevents the group from being bi-orderable. 
A group $F$ is called {\it bi-orderable} if it admits a total order $<$ such that $agb<ahb$ holds for any $g,h,a,b\in F$ whenever $g<h$ holds i.e the order $<$ is invariant under multiplying elements from left and right. Such an order is called {\it bi-order}. 
Suppose that $F$ has a generalized torsion element $g$ and take $h_{1},\dots,h_{n}$ in the definition. 
If $F$ admitted a bi-order $<$, then $g^{h_k}<1_{F}$ for all $k$ or $g^{h_k}>1_{F}$ for all $k$ would hold. 
This would imply that $g^{h_1}\cdot g^{h_{2}}\cdots g^{h_n}<1_{F}$ or $g^{h_1}\cdot g^{h_{2}}\cdots g^{h_n}>1_{F}$ would hold. 
This leads a contradiction. 
Hence existing of a generalized torsion element prevents a group from being bi-orderable. 

\begin{defini}
For an element $g\in F$, the semigroup consisting of non-empty finite products of conjugates of $g$ is denoted by $\left<\left<g\right>\right>^{+}$. 
It can be checked easily that a non-trivial element $g$ is a generalized torsion element if and only if $1_{F}\in \left<\left<g\right>\right>^{+}$.
\end{defini}

\begin{lem}\label{lemitomotegiteragaito}{\rm (Lemma 4.1 in \cite{itomotegiteragaito})}\\
Let $g,h$ and $x$ be elements of $F$. Then the following holds.
\begin{enumerate}
 \item $g^{n}h^{n}\in \left<\left<gh\right>\right>^{+}$ for all $n>0$.
 \item If $gh\in \left<\left<x\right>\right>^{+}$, then $g^{n}h^{n} \in \left<\left<x\right>\right>^{+}$ for all $n>0$.
 \item If $[g,h]\in \left<\left<x\right>\right>^{+}$, then $[g^{n},h^{m}]\in \left<\left<x\right>\right>^{+}$ for all $n,m>0$.
\end{enumerate}
\end{lem}
\begin{proof} (From \cite{itomotegiteragaito})\\
The first follows from the equality $g^{n}h^{n}=(gh)^{g^{-(n-1)}}(gh)^{g^{-(n-2)}}\cdots(gh)$. \\
The second follows from the first and the fact that if $gh\in \left<\left<x\right>\right>^{+}$, then $\left<\left<gh\right>\right>^{+}\subset \left<\left<x\right>\right>^{+}$.\\
For the third, assume that $[g,h]=g^{-1}(h^{-1}gh)\in \left<\left<x\right>\right>^{+}$. 
%Then $\left<\left<[g,h]\right>\right>^{+}=\left<\left<g^{-1}(h^{-1}gh)\right>\right>^{+}\subset \left<\left<x\right>\right>^{+}$. 
Apply the second to see that $(g^{-n}h^{-1}g^{n})h=g^{-n}(h^{-1}g^{n}h)\in \left<\left<x\right>\right>^{+}$. 
Then apply the second again to see that $[g^n,h^m]=g^{-n}h^{-m}g^{n}h^{m}=(g^{-n}h^{-m}g^{n})h^{m}\in \left<\left<x\right>\right>^{+}$. 
\end{proof}

\begin{lem}\label{lemforlongitude}
Let $g$ and $h$ be elements of $F$, and $n$ and $m$ positive integers. \\
Then $\{(hg)^{-m}(gh)^{m}\}^{n}\{(hg)^{m}(gh)^{-m}\}^{n}\in \left<\left<[gh,hg]\right>\right>^{+}$. 
\end{lem}

\begin{proof}
By the third of Lemma~\ref{lemitomotegiteragaito}, we have 
\begin{align}
(hg)^{-m}(gh)^{m}(hg)^{m}(gh)^{-m}=[(hg)^{m},(gh)^{-m}]\in \left<\left< [hg,(gh)^{-1}]\right>\right>^{+}=\left<\left< (gh)[gh,hg](gh)^{-1}\right>\right>^{+}=\left<\left< [gh,hg]\right>\right>^{+}. \notag
\end{align}
Moreover, by the second of Lemma~\ref{lemitomotegiteragaito}, we have $\{(hg)^{-m}(gh)^{m}\}^{n}\{(hg)^{m}(gh)^{-m}\}^{n}\in \left<\left<[gh,hg]\right>\right>^{+}$. 
\end{proof}

\section{Double twist knots and their $0$-surgeries}\label{secdoubletwist}
\subsection{Double twist knots}
For non-zero integers $p$ and $q$, $K_{p,q}$ denotes the knot which is represented as $C[2p,2q]$ in the Conway notation. 
This $K_{p,q}$ is called the {\it double twist knot of type} $(p,q)$. 
See Figure~\ref{Kpq}, which we call the standard diagram of $K_{p,q}$. 
In Figure~\ref{Kpq}, $(-n)$-left twists (or right twists) represents $n$-right twists (or left twists, respectively). 
For example, $K_{1,1}$ is the figure eight knot and $K_{1,-1}$ is the right-handed trefoil. 
Note that $K_{p,q}$ is isotopic to $K_{-q,-p}$, that $K_{p,q}$ is the mirror of $K_{-p,-q}$ and $K_{q,p}$ (see Figure~\ref{isotopymirror}), that the bridge number of $K_{p,q}$ is two and that the genus of $K_{p,q}$ is one.

\begin{figure}[htbp]
 \begin{center}
  \includegraphics[width=120mm]{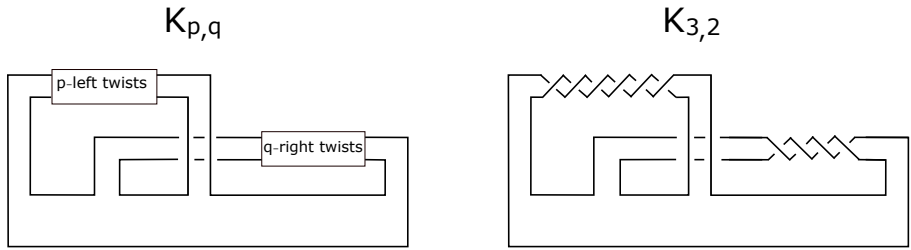}
 \end{center}
 \caption{left: the standard diagram of $K_{p,q}$, \ \ \ right: $K_{3,2}$}
 \label{Kpq}
\end{figure}

\begin{figure}[htbp]
 \begin{center}
  \includegraphics[width=140mm]{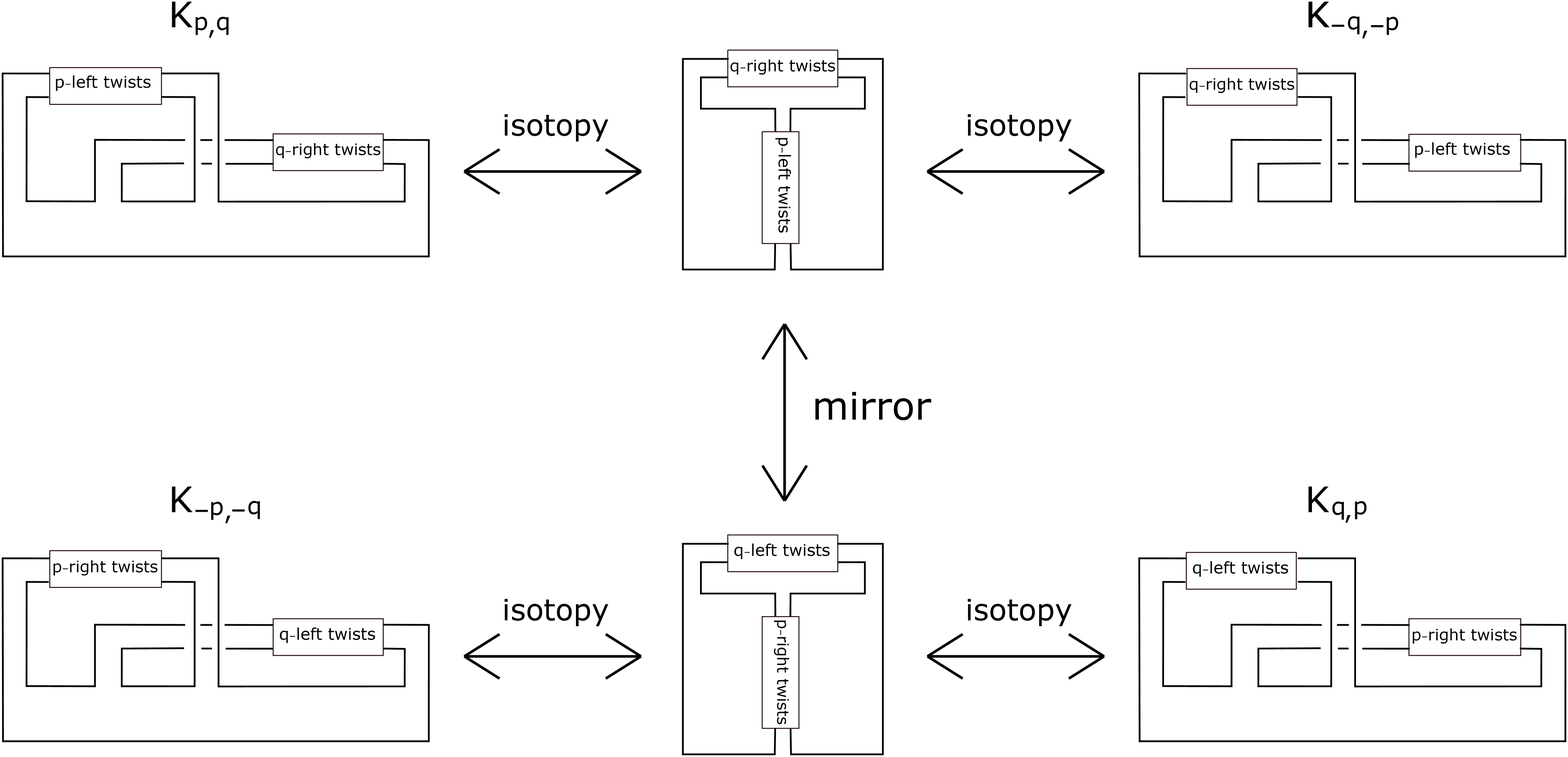}
 \end{center}
 \caption{Correspondence among $K_{p,q}$, $K_{-p,-q}$, $K_{-p,-q}$ and $K_{q,p}$. }
 \label{isotopymirror}
\end{figure}

In the standard diagram of $K_{p,q}$, we fix a Seifert surface $T'$ of genus one as the shaded surface in Figure~\ref{seifertsurf}. 
Let $E_{p,q}$ denote the closure of the complement of $K_{p,q}$ in $S^{3}$.  
There is a useful presentation of the fundamental group of $E_{p,q}$, so called the Lin presentation: 
\begin{align}
 \pi_{1}(E_{p,q}) = \left<a, b, t \mid ta^{p}t^{-1}=b^{-1}a^{p}, \ tb^{-q}a^{-1}t^{-1}=b^{-q} \right> \label{eqlinpresentation}
\end{align}
In this presentation, $a$, $b$ and $t$ are based loops in $E_{p,q}$ as in Figure~\ref{linpresentation}. 
This presentation is obtained by noting that the complement of the Seifert surface $T'$ is a handlebody whose fundamental group is generated by loops $a$ and $b$, that the push ups of  loops $x$ and $y$ on $T'$ is $a^p$ and $b^{-q}a^{-1}$, respectively and that the push downs of  loops $x$ and $y$ on $T'$ is $b^{-1}a^{p}$ and $b^{-q}$, respectively. 
In this presentation, the meridian of $K_{p,q}$ corresponds to $t$ and the canonical longitude of it corresponds to $[b^{q},a^{p}]$. 

\begin{figure}[htbp]
 \begin{center}
  \includegraphics[width=70mm]{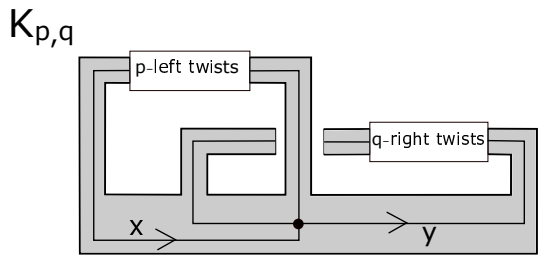}
 \end{center}
 \caption{Seifert surface $T'$ of $K_{p,q}$ and based loops on $T'$}
 \label{seifertsurf}
\end{figure}

\begin{figure}[htbp]
 \begin{center}
  \includegraphics[width=70mm]{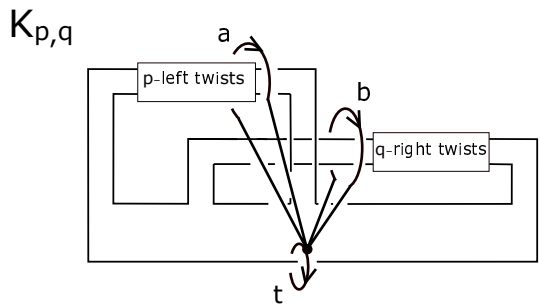}
 \end{center}
 \caption{the generators of the Lin presentation of $\pi_{1}(E_{p,q})$}
 \label{linpresentation}
\end{figure}

\subsection{The $0$-surgeries along double twist knots}
Let $K_{p,q}(0)$ denote the 3-manifold obtained by the $0$-surgery along $K_{p,q}$. 
In $K_{p,q}(0)$, the boundary of the (minimal genus) Seifert surface $T'$ is capped-off by a disk. 
We call this resultant closed torus $T$. 
According to \cite{gabai}, $K_{p,q}(0)$ is irreducible and $T$ is incompressible. 
From the presentation \eqref{eqlinpresentation}, the fundamental group of $K_{p,q}(0)$, denoted by $G_{p,q}$, can be computed as: 
\begin{align}\label{stdrep}
 G_{p,q}= \pi_{1}\left(K_{p,q}(0)\right) = \left<a, b, t \mid ta^{p}t^{-1}=b^{-1}a^{p}, \ tb^{-q}a^{-1}t^{-1}=b^{-q}, \  [b^{q},a^{p}]=1 \right>.
\end{align}
Moreover, we consider $\overline{K_{p,q}(0)\setminus T}$, the closure  of the complement of $T$ in $K_{p,q}(0)$. 
This is obtained from the closure of the complement of the Seifert surface $T'$ of $K_{p,q}$ in $S^{3}$, which is a handlebody of genus two by attaching a $2$-handle along the longitude of $K_{p,q}$. 
Following this construction, the fundamental group of $\overline{K_{p,q}(0)\setminus T}$, denoted by $\Gamma_{p,q}$, can be computed as: 
\begin{align}
 \Gamma_{p,q}= \pi_{1}\left( \overline{K_{p,q}(0)\setminus T}\right) = \left<a, b\mid [b^{q},a^{p}]=1 \right>. \label{gammapq}
\end{align}
The boundary of $\overline{K_{p,q}(0)\setminus T}$ consists of two tori. 
We call the one coming from the front side of $T'$ $T_{+}$ and call the other $T_{-}$. 
We fix presentations of $\pi_{1}(T_{\pm})$, denoted by $\Gamma_{\pm}$ as follows, where $X_{\pm}$ and $Y_{\pm}$ are coming from $x$ and $y$ in Figure~\ref{seifertsurf}. 
\begin{align}
\Gamma_{+}&=\pi_{1}(T_{+})=\left< X_{+}, Y_{+} \mid [X_{+}, Y_{+}]=1 \right>  \\ \notag \\
\Gamma_{-}&=\pi_{1}(T_{-})=\left< X_{-}, Y_{-} \mid [X_{-}, Y_{-}]=1 \right> 
\end{align}
Then we take homomorphisms $\iota _{+}: \Gamma_{+}\longrightarrow \Gamma$ and $\iota _{-}: \Gamma_{-}\longrightarrow \Gamma$ which are induced by the inclusions such that $\iota_{+}(X_{+})=a^{p}$, $\iota_{+}(Y_{+})=b^{-q}a^{-1}$, $\iota_{-}(X_{-})=b^{-1}a^{p}$ and $\iota_{-}(Y_{-})=b^{-q}$ hold. 
Note that since $K_{p,q}(0)$ is irreducible and $T$ is incompressible, $\overline{K_{p,q}(0)\setminus T}$ is irreducible and has incompressible boundary. 

\subsection{Another presentation}\label{subsecanother}
For our use, we give another presentations of the fundamental groups of $E_{p,q}$ and $K_{p,q}(0)$. 
By changing the presentation \eqref{eqlinpresentation} as follows, we get a presentation of the fundamental groups of $E_{p,q}$.
\begin{align}
  \pi_{1}(E_{p,q}) &= \left<a, b, t \mid ta^{p}t^{-1}=b^{-1}a^{p}, \ tb^{-q}a^{-1}t^{-1}=b^{-q}  \right> \notag \\ 
   &= \left<a, b, t \mid ta^{p}t^{-1}=b^{-1}a^{p}, \ a=t^{-1}b^{q}tb^{-q}  \right> \notag \\ 
   &= \left<b, t \mid t(t^{-1}b^{q}tb^{-q})^{p}t^{-1}=b^{-1}(t^{-1}b^{q}tb^{-q})^{p} \right> \notag \\
   &= \left<b, t, x, y \mid t(t^{-1}b^{q}tb^{-q})^{p}t^{-1}=b^{-1}(t^{-1}b^{q}tb^{-q})^{p},\ x=bt, \ y=t^{-1} \right> \notag \\
    &= \left<b, t, x, y \mid t(t^{-1}b^{q}tb^{-q})^{p}t^{-1}=b^{-1}(t^{-1}b^{q}tb^{-q})^{p},\ b=xy, \ t=y^{-1} \right> \notag \\
    &= \left<x, y \mid y^{-1} \{ (yx)^{q}(xy)^{-q})\}^{p}  y=y^{-1}x^{-1}\{ (yx)^{q}(xy)^{-q})\}^{p} \right> \notag \\
    &= \left<x, y \mid \{ (yx)^{q}(xy)^{-q})\}^{p}  y=x^{-1}\{ (yx)^{q}(xy)^{-q})\}^{p} \right> \label{anotherpresentation1}
\end{align}  
In the deformation above, the meridian $t$ is changed into $y^{-1}$ and the longitude $[b^{q},a^{p}]$ is changed into $\{ (yx)^{-q}(xy)^{q})\}^{p}\{ (yx)^{q}(xy)^{-q})\}^{p}$. 
Then we get another presentation of $G_{p,q}$, the fundamental group of $K_{p,q}(0)$ as follows. 
\begin{align}
 G_{p,q} = \left<x, y \mid \{ (yx)^{q}(xy)^{-q})\}^{p}  y=x^{-1}\{ (yx)^{q}(xy)^{-q})\}^{p}, \ \{ (yx)^{-q}(xy)^{q})\}^{p}\{ (yx)^{q}(xy)^{-q})\}^{p}=1 \right> \label{anotherpresentation2}
\end{align}

\begin{rmk}
The presentations \eqref{anotherpresentation1} and \eqref{anotherpresentation2} are obtained directly by noting that the tunnel number of $K_{p,q}$ is one.
\end{rmk}

\begin{lem}\label{lemcandidate}
Let $p$ be a positive integer and $q$ a non-zero integer. 
Then in $G_{p,q}$ with presentation \eqref{anotherpresentation2}, %some products of some conjugates of $[xy,yx]$ is the identity element. 
$1\in \left<\left< [xy,yx] \right>\right>^{+}$ holds. 
\end{lem}
\begin{proof}
For positive $q$, considering the second relation of \eqref{anotherpresentation2} and applying Lemma~\ref{lemforlongitude} by replacing $g$ and $h$ with $x$ and $y$ respectively, we have
\begin{align}
1=\{ (yx)^{-q}(xy)^{q})\}^{p}\{ (yx)^{q}(xy)^{-q})\}^{p}\in \left<\left<[xy,yx]\right>\right>^{+}. \notag
\end{align}
For negative $q$, considering the second relation of \eqref{anotherpresentation2} and applying Lemma~\ref{lemforlongitude} by replacing $g$ and $h$ with $y^{-1}$ and $x^{-1}$ respectively, we have
\begin{align}
1&=\{ (yx)^{-q}(xy)^{q})\}^{p}\{ (yx)^{q}(xy)^{-q})\}^{p}=\{ (x^{-1}y^{-1})^{q}(y^{-1}x^{-1})^{-q})\}^{p}\{ (x^{-1}y^{-1})^{-q}(y^{-1}x^{-1})^{q})\}^{p} \notag \\
 &\in \left<\left<[y^{-1}x^{-1},x^{-1}y^{-1}]\right>\right>^{+}=\left<\left<[(xy)^{-1},(yx)^{-1}]\right>\right>^{+}=\left<\left<[xy,yx]\right>\right>^{+}. \notag
\end{align}
\end{proof}

Note that the lemma above does not imply that $[xy,yx]$ is a generalized torsion element. 
It may be the trivial element.

\section{The JSJ-decompositions of 3-manifolds obtained by the $0$-surgeries along double twist knots}\label{secjsj}
Let $M$ be a prime compact 3-manifold with (possibly empty) incompressible boundary. 
There exists a set $\mathcal{T}$ consisting of pairwise non-parallel disjoint essential tori in $M$ such that
\begin{itemize}
\item each component obtained by cutting along the tori in $\mathcal{T}$ is either a Seifert manifold or an atoroidal manifold, and
\item there are no proper subset of $\mathcal{T}$ satisfying the above condition.
\end{itemize}
This decomposition is called the {\it JSJ-decomposition} of $M$ \cite{jacoshalen} \cite{johannson}. 
It is known that the set $\mathcal{T}$ is unique for a given $M$. 
%By the Geometrization conjecture, proved by [Perelman], atoroidal components are hyperbolic manifolds. 
In this section, we give the JSJ-decomposition of the 3-manifold $K_{p,q}(0)$ obtained by the $0$-surgery along the double twist knot $K_{p,q}$.

\subsection{Components for the decomposition}\label{subseccomponents}
We will define 3-manifolds $M(k)$, which appear in the JSJ-decompositions of 3-manifolds obtained by the $0$-surgeries along double twist knots for a non-zero integer $k$. 
Let $A$ be an annulus. 
Take a point $*\in {\rm Int}A$, in the interior of $A$. 
Consider 3-manifold $A\times S^1$. 
For a non-zero integer $k$, $M(k)$ denotes the 3-manifold obtained by $k$-surgery along $\{*\}\times S^1$ from $A\times S^1$. 
We take a longitude of $\{*\}\times S^1$ for the surgery such that it is isotopic to $\{*'\}\times S^1$ for $\{*'\} \in A\setminus \{*\}$ in $\left( A\setminus \{*\} \right)\times S^1$. 
Note that $M(k)$ is a Seifert manifold whose regular fiber comes from the $S^{1}$ factor of $A\times S^1$, that the base orbifold of this fiber structure is annulus with one (or zero) exceptional point, and that $M(k)$ is irreducible and has incompressible boundary. 
Moreover, note that $M(\pm1)$ is $S^{1}\times S^{1}\times [0,1]$ and that $M(k)$ admits the unique fiber structure unless $k=\pm1$. See \cite{jaco} for example. 

\begin{lem}\label{lemmk}
For non-zero integers $k_1$ and $k_2$, $M(k_1)$ and $M(k_2)$ are homeomorphic if and only if $|k_1|=|k_2|$.
\end{lem}
\begin{proof}
``If'' part is clear. Note that if $k_1=-k_2$, the homoeomorphism may be orientation reversing. 
We prove ``only if'' part. 
Suppose that $M(k_1)$ and $M(k_2)$ are homeomorphic. 
If one of $k_1$ and $k_2$ is $\pm1$, then the manifold is $S^{1}\times S^{1}\times [0,1]$ and the other is also $\pm1$. 
If neither $k_1$ nor $k_2$ is $\pm1$, then the manifolds admit the unique fiber structures. 
Thus the homeomorphism sends regular and singular fibers to regular and singular fibers, respectively. 
Note that in $M(k_1)$, a regular fiber near the singular fiber spirals around the singular fiber $|k_1|$ times and that in $M(k_2)$, a regular fiber near the singular fiber spirals around the singular fiber $|k_2|$ times. 
Thus $|k_1|=|k_2|$ must hold. 
\end{proof}

Let $X_{p,q}$ be a manifold obtained by gluing $M(-q)$ and $M(p)$ as follows: 
Take base points and based loops in $M(-q)$ and $M(p)$ as in Figure~\ref{annuli}. 
In Figure~\ref{annuli}, $h_1$ and $h_2$ are based loops which are also regular fibers in $M(-q)$ and $M(p)$ pointing into the front side of this paper, and $T_1$, $T_2$, $\partial_{+}$ and $\partial_{-}$ are boundary tori. 
We can see that $\pi_{1}\left(M(-q)\right)=\left< \alpha_1, \beta_1, h_1 \mid [\alpha_1, h_1]=[\beta_1, h_1]=h_{1}{\alpha_{1}}^{-q}=1\right>$ and \\
$\pi_{1}\left(M(p)\right)=\left< \alpha_2, \beta_2, h_2 \mid [\alpha_2, h_2]=[\beta_2, h_2]=h_{2}{\alpha_{2}}^{p}=1\right>$. 
Paste $M(-q)$ and $M(p)$ along $T_1$ and $T_2$ so that the pair $(\beta_1, h_1)$ is identified with $(h_2, {\beta_2}^{-1})$. 
This resultant manifold is denoted by $X_{p,q}$. 
The boundary of $X_{p,q}$ coming from $\partial_{\pm}$ is also denoted by $\partial_{\pm}$. 
Note that $X_{p,q}$ is also irreducible and has incompressible boundary. 

\begin{figure}[htbp]
 \begin{center}
  \includegraphics[width=110mm]{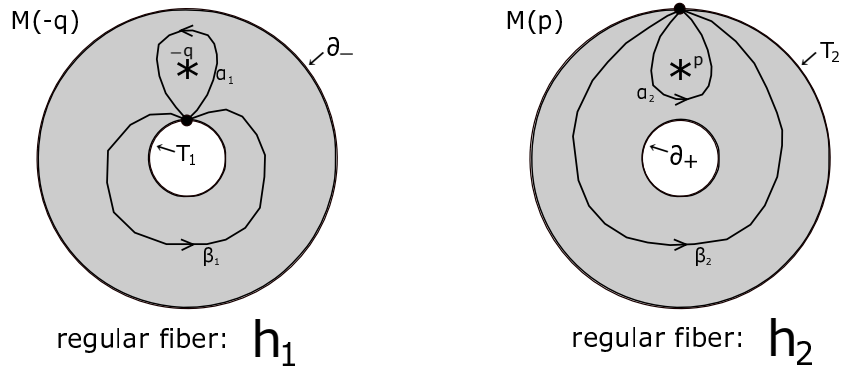}
 \end{center}
 \caption{$M(-q)$ and $M(p)$}
 \label{annuli}
\end{figure}

Set $G$, $G_{+}$ and $G_{-}$ to be the fundamental groups of $X_{p,q}$, $\partial _{+}$ and $\partial_{-}$, respectively. 
Let $j_{\pm}:G_{+}\longrightarrow G$ be (some) homomorphisms induced by the inclusions. 
Then we can compute and fix presentations of $G$ and $G_{\pm}$ as follows.
\begin{align}
G &= \left<\alpha_1, \beta_1, h_1, \alpha_2, \beta_2, h_2 \mid [\alpha_1, h_1]=[\beta_1, h_1]=h_{1}{\alpha_{1}}^{-q}=[\alpha_2, h_2]=[\beta_2, h_2]=h_{2}{\alpha_{2}}^{p}=\beta_{1}{h_{2}}^{-1}=\beta_{2}h_1=1 \right> \notag \\ 
   &= \left<\alpha_1, \alpha_2 \mid [{\alpha_1}^{q}, {\alpha_2}^p]=1 \right>  \\ \notag \\
G_{+}&=\left< x_{+}, y_{+} \mid [x_{+}, y_{+}]=1 \right>  \\ \notag \\
G_{-}&=\left< x_{-}, y_{-} \mid [x_{-}, y_{-}]=1 \right> 
\end{align}   

We take $x_{\pm}$ and $y_{\pm}$ so that $j_{+}(x_{+})=(\beta_{2}{\alpha_2}^{-1}=){\alpha_{1}}^{-q}{\alpha_2}^{-1}$, $j_{+}(y_{+})=(h_2=){\alpha_2}^{-p}$, $j_{-}(x_{-})=(\alpha_1\beta_1=)\alpha_1{\alpha_2}^{-p}$ and $j_{-}(y_{-})=(h_1=)\alpha_{1}^q$ hold. 

\begin{lem}\label{lemhomeoxpq}
$\overline{K_{p,q}(0)\setminus T}$ is homeomorphic to $X_{p,q}$. 
\end{lem}
\begin{proof}
Recall that $G$ is the fundamental group of $X_{p,q}$ and that $\Gamma$ is that of $\overline{K_{p,q}\setminus T}$. 
We use the presentation \eqref{gammapq} for $\Gamma$. 
Set $\phi$, $\phi_{+}$, $\phi_{-}$ and $f$ to be isomorphisms satisfying: 
\begin{itemize}
\item $\phi:G\longrightarrow \Gamma$ maps $\alpha_1$ and $\alpha_2$ to $b^{-1}$ and $a^{-1}$ respectively,
\item $\phi_{+}:G_{+}\longrightarrow \Gamma_{+}$ maps $x_{+}$ and $y_{+}$ to ${Y_{+}}^{-1}$ and $X_{+}$ respectively, 
\item $\phi_{-}:G_{-}\longrightarrow \Gamma_{-}$ maps $x_{-}$ and $y_{-}$ to $X_{-}$ and $Y_{-}$ respectively, and
\item $f:\Gamma \longrightarrow \Gamma$ maps $a$ and $b$ to $a$ and $a^{-1}ba$ respectively. 
\end{itemize}
Then the following two diagrams are commute, where $id_{\Gamma}$ denotes the identity map of $\Gamma$.
\[
 \xymatrix{
   \Gamma_{+} \ar[r]^{\iota_{+}} & \Gamma \ar[r]^{f}  & \Gamma &  & \Gamma_{-} \ar[r]^{\iota_{-}} & \Gamma \ar[r]^{id_{\Gamma}}  & \Gamma\\
   G_{+}\ar[u]^{\phi_{+}} \ar[r]^{j_{+}} &  G\ar[ur]_{\phi} & & &  G_{-}\ar[u]^{\phi_{-}} \ar[r]^{j_{-}} &  G\ar[ur]_{\phi} 
}
\]
Therefore, $\phi$ (with other isomorphisms) is an isomorphism preserving the peripheral structure. 
Note that $X_{p,q}$ and $\overline{K_{p,q}(0)\setminus T}$ are irreducible and that they have essential surfaces since 3-manifold whose boundary component other than spheres are not empty has an essential surface. 
By a result of Waldhausen \cite{waldhausen}, this isomorphism is induced by a homeomorphism between $X_{p,q}$ and $\overline{K_{p,q}(0)\setminus T}$. 
\end{proof}

\begin{lem}\label{lemnonseifert}
$X_{p,q}$ is not a Seifert manifold unless $p=\pm1$ or $q=\pm1$. 
\end{lem}
\begin{proof}
Suppose that $p\neq \pm1$ and $q\neq \pm1$. 
Suppose that $X_{p,q}$ is a Seifert manifold for the contrary. 
$X_{p,q}$ has an incompressible torus $\tilde{T}$, which is coming from the attaching torus of $M(p)$ and $M(-q)$. 
It is known that (see \cite{jaco}) if $F$ is a connected incompressible surfaces in a Seifert manifold $M$, then either
\begin{itemize}
\item $F$ is a boundary parallel annulus, 
\item $M$ is a surface bundle over $S^1$ and $F$ is a fiber surface,
\item $F$ splits $M$ into two 3-manifolds $M_1$ and $M_2$, both of which are twisted $I$-bundles over non-orientable surfaces or
\item $F$ is annulus or torus and isotopic to the one foliated by Seifert fibers of $M$
\end{itemize}
holds. 
Apply this fact to $\tilde{T}$. 
Since $\tilde{T}$ is a closed torus and $X_{p,q}$ has boundary, the first and the second cannot hold. 
Torus $\tilde{T}$ splits $X_{p,q}$ into $M(p)$ and $M(-q)$. 
Note that the boundary of each split component consists of two tori. 
Suppose that $M(p)$ is a twisted $I$-bundle over a non-orientable surface $N$. 
Note that $N$ must have boundary since $M(p)$ has boundary. 
Then the Euler number of the boundary of $M(p)$ is twice as many as that of $N$. 
Thus $N$ is a M${\rm \ddot{o}}$bius band. 
However the twisted $I$-bundle over a M${\rm \ddot{o}}$bius band is a solid torus, not $M(p)$. 
Thus the third cannot hold. 
If $\tilde{T}$ can be foliated by Seifert fibers of $X_{p,q}$, then the fiber structure descends to the split components $M(p)$ and $M(-q)$, and the foliations on the cut ends coming from $\tilde{T}$ are identical. 
Since $p\neq \pm1$ and $q\neq \pm1$, these fiber structures are the same as the ones defined in the construction of $M(k)$. 
However, in the construction of $X_{p,q}$, the foliations on $\tilde{T}$ are not identical. 
Thus the fourth cannot hold, and we conclude that $X_{p,q}$ is not a Seifert manifold unless $p=\pm1$ or $q=\pm1$.
\end{proof}

\subsection{The JSJ-decompositions of 3-manifolds obtained by the $0$-surgeries along double twist knots}

\begin{prop}\label{propjsj}(The JSJ-decomposition of $K_{p,q}(0)$)\\
Let $p$ and $q$ be non-zero integers. 
The set of decomposing tori $\mathcal{T}$ of the JSJ-decomposition of $K_{p,q}(0)$ is
\begin{itemize}
\item empty if $(p,q)=(1,-1)$ or $(-1,1)$, and then $K_{p,q}(0)$ is a Seifert manifold, 
\item $\{T\}$, where $T$ is the result of capping off the Seifert surface $T'$ of $K_{p,q}$ in Figure~\ref{seifertsurf} if $p=\pm1$ or $q=\pm1$ and $(p,q)$ is neither $(1,-1)$ nor $(-1,1)$, and then the resulting component of splitting is homeomorphic to a Seifert manifold $M(pq)$, and 
\item $\{T,\tilde{T}\}$, where $\tilde{T}$ is the torus in $X_{p,q}$, homeomorpic to $\overline{K_{p,q}(0)\setminus T}$, splitting it into Seifert manifolds $M(p)$ and $M(-q)$ if $p\neq \pm1$ and $q\neq \pm1$. 
\end{itemize}
\end{prop}
\begin{proof}
In \cite{ichihara_dae}, it is proved that a Montesinos knot admits a Dehn surgery yielding a troidal Seifert 3-manifold if and only if the knot is trefoil and the surgery slope is $0$. 
Thus we know that $K_{p,q}(0)$ is a Seifert manifold and the set of decomposing tori is empty if $(p,q)=(1,-1)$ or $(-1,1)$ and a non-Seifert manifold otherwise. 
Assume that $(p,q)$ is neither $(1,-1)$ nor $(-1,1)$. 
Then $K_{p,q}(0)$ is a non-Seifert manifold and having essential torus $T$, and the set of decomposing tori must contain $T$. 
By Lemma~\ref{lemhomeoxpq}, the result of cutting $K_{p,q}(0)$ along $T$ is homeomorphic to $X_{p,q}$. 
This $X_{p,q}$ is a Seifert manifold homeomorphic to $M(pq)$ if $p=\pm1$ or $q=\pm1$, and otherwise a non-Seifert manifold by Lemma~\ref{lemnonseifert} and having essential torus $\tilde{T}$, which splits $X_{p,q}$ into two Seifert manifolds $M(p)$ and $M(-q)$. 
This completes the proof.
\end{proof}

Using the JSJ-decomposition above, we get the following:
\begin{prop}\label{prophomeo0surg}
Let $K_{p,q}$ and $K_{p',q'}$ be two double twist knots. 
Then $K_{p,q}(0)$ and $K_{p',q'}(0)$ are homeomorphic if and only if $(p',q')=(p,q)$, $(q,p)$, $(-p,-q)$ or $(-q,-p)$. 
\end{prop}
\begin{proof}
``If'' part is clear since $K_{p',q'}$ is isotopic to $K_{p,q}$ or its mirror in such condition of $(p',q')$. 
We prove ``only if'' part.
Since the $0$-surgery of a knot determines the Alexander polynomial of the knot and the Alexander polynomial $\Delta_{K_{p,q}}(t)$ of $K_{p,q}$ is $-pqt^{2}+(2pq+1)t-pq$, the equation $pq=p'q'$ must hold. 
If $(p,q)$ is $(1,-1)$ or $(-1,1)$, the set of decomposing tori of $K_{p,q}(0)$ is empty. 
This also holds for that of $K_{p',q'}(0)$. 
By Proposition~\ref{propjsj}, this implies that $(p',q')=(1,-1)$ or $(-1,1)$. 
We assume that $(p,q)$ is neither $(1,-1)$ nor $(-1,1)$. 
If $p=\pm1$ or $q=\pm1$, the set of decomposing tori consists of one element. 
This also holds for that of $K_{p',q'}(0)$. 
By Proposition~\ref{propjsj}, this implies that $p'=\pm1$ or $q'=\pm1$. 
Combining this with the equation $pq=p'q'$, we conclude that $(p',q')=(p,q)$, $(q,p)$, $(-p,-q)$ or $(-q,-p)$. 
Assume that $p\neq \pm1$ and $q\neq \pm1$ next. 
Then the set of decomposing tori for $K_{p,q}(0)$ consists of two tori and this also holds for that of $K_{p',q'}(0)$. 
Moreover, the sets $\{M(p),M(-q)\}$ and $\{M(p'), M(-q')\}$ of the resulting pieces of splitting must identical. 
By Lemma~\ref{lemmk}, $p=\pm p'$ or $\pm q'$. 
Combining this with the equation $pq=p'q'$, we conclude that $(p',q')=(p,q)$, $(q,p)$, $(-p,-q)$ or $(-q,-p)$.
\end{proof}

\begin{rmk}
There are many inequivalent knots which produce homeomorphic 3-manifolds by the $0$-surgeries \cite{brakes} \cite{osoinach} \cite{yasui} \cite{manolescupiccirillo}. 
On the other hand, there are many knots such that the $0$-surgery along them never homeomorphic to those along inequivalent knots \cite{baldwinsiviek1} \cite{baldwinsiviek2}. 
In these cases, it is said that {\it $0$ is the characterizing slope for the knot}. 
Note that the proposition above does not state that $0$ is the characterizing slope for a double twist knot. 
There may be a knot in a class of knots other than double twist knots a surgery along which produces 3-manifold homeomorphic to the 3-manifold obtained by the $0$-surgery along a double twist knot. 
\end{rmk}

\section{A proof of Theorem~\ref{mainthm}}\label{secprfofthm}
Take two pair of non-zero integers $(p,q)$ and $(p',q')$ such that $pq=p'q'$ and $p\neq \pm p', \pm q'$. 
Recall that $K_{p,q}$ is the mirror of $K_{-p,-q}$. 
And note that the knot group of a knot is isomorphic to that of the mirror of the knot and that this isomorphism extends to an isomorphism between the fundamental groups of the 3-manifolds obtained by the $0$-surgeries along the knot and its mirror. 
Thus we assume that $p$ and $p'$ are positive. 
We use the presentations of type \eqref{anotherpresentation2} for $G_{p,q}$ and $G_{p',q'}$, the fundamental groups of $K_{p,q}(0)$ and $K_{p',q'}(0)$ respectively. 
As stated at Lemma~\ref{lemcandidate}, $[xy,yx]$ is a generalized torsion element in $G_{p,q}$ unless $[xy,yx]=1$ holds. 
Suppose that $[xy,yx]=1$ holds in $G_{p,q}$ i.e $xy$ and $yx$ commute in $G_{p,q}$. 
Then we can deform the presentation \eqref{anotherpresentation2} as follows. 
Note that the second relation of the presentation \eqref{anotherpresentation2} follows from $[xy,yx]=1$ by Lemma~\ref{lemcandidate}. 
\begin{align}
 G_{p,q} &= \left<x, y \mid \{ (yx)^{q}(xy)^{-q})\}^{p}  y=x^{-1}\{ (yx)^{q}(xy)^{-q})\}^{p}, \ \{ (yx)^{-q}(xy)^{q})\}^{p}\{ (yx)^{q}(xy)^{-q})\}^{p}=1 \right> \notag \\
           &=\left<x, y \mid \{ (yx)^{q}(xy)^{-q})\}^{p}  y=x^{-1}\{ (yx)^{q}(xy)^{-q})\}^{p}, \ \{ (yx)^{-q}(xy)^{q})\}^{p}\{ (yx)^{q}(xy)^{-q})\}^{p}=1,\ [xy,yx]=1 \right> \notag \\
           &=\left<x, y \mid \{ (yx)^{q}(xy)^{-q})\}^{p}  y=x^{-1}\{ (yx)^{q}(xy)^{-q})\}^{p}, \ [xy,yx]=1 \right> \notag \\
           &=\left<x, y \mid \{ (yx)(xy)^{-1})\}^{pq}  y=x^{-1}\{ (yx)(xy)^{-1})\}^{pq}, \ [xy,yx]=1 \right> \notag
\end{align}
Similarly, if $[xy,yx]=1$ holds in $G_{p',q'}$, then we have the following presentation. 
\begin{align}
 G_{p',q'} &=\left<x, y \mid \{ (yx)(xy)^{-1})\}^{p'q'}  y=x^{-1}\{ (yx)(xy)^{-1})\}^{p'q'}, \ [xy,yx]=1 \right> \notag
\end{align}
Therefore, if $[xy,yx]=1$ holds simultaneously in $G_{p,q}$ and $G_{p',q'}$, then $G_{p,q}$ and $G_{p',q'}$ are isomorphic. 
Note that $K_{p,q}(0)$ and $K_{p',q'}(0)$ is irreducible and have essential surfaces, which are obtained by capping off a Seifert surfaces of minimal genus. 
By a result of Waldhausen \cite{waldhausen} (for closed 3-manifold), stating that every isomorphism between the fundamental groups of closed 3-manifolds which have essential surfaces is induced by a homeomorphism between the 3-manifolds, we know that $K_{p,q}(0)$ and $K_{p',q'}(0)$ are homeomorphic. 
However, this is impossible in our condition of $(p,q)$ and $(p',q')$ by Proposition~\ref{prophomeo0surg}. 
Therefore, $[xy,yx]\neq 1$ in at least one of $G_{p,q}$ and $G_{p',q'}$. 

\begin{rmk}
As one can see in the proof above, we say that $[xy,yx]$ is not a generalized torsion element in at most one of $G_{p,q}$ and $G_{p',q'}$. 
There may be another generalized torsion element in $G_{p,q}$ even if $[xy,yx]$ is not a generalized torsion element. 
Later we construct an (at most) infinite family of elements in $G_{p,q}$. 
And show that if $G_{p,q}$ admits a generalized torsion element, then $G_{p,q}$ admits a torsion element or admits a generalized torsion element in the constructed family. 
For the family of elements, see the last paragraph of Section~\ref{secprfofthm2} for $pq<0$ and Remark~\ref{candidates} for $pq>0$. 
\end{rmk}

\section{A proof of Theorem~\ref{mainthm2}}\label{secprfofthm2}
In this section, we search a generalized torsion element in the fundamental group $G_{p,-q}$ of the 3-manifold $K_{p,-q}(0)$ obtained by the $0$-surgery along a double twist knot $K_{p,-q}$ for positive integers $p$ and $q$. 
For $p=1$, the generalized torsion element in the knot group of $K_{1,-q}$ constructed in \cite{teragaito1} remains to be a generalized torsion element after the $0$-surgery. 
Though we cannot find a generalized torsion element in the knot group of $K_{p,-q}$ for general $p$ so far, we can see that there exists a generalized torsion element in $G_{p,-q}$. 
We use the presentation (\ref{stdrep}) for $G_{p,-q}$. 
\begin{align}
 G_{p,-q}= \pi_{1}\left(K_{p,-q}(0)\right) = \left<a, b, t \mid ta^{p}t^{-1}=b^{-1}a^{p}, \ tb^{q}a^{-1}t^{-1}=b^{q}, \  [b^{-q},a^{p}]=1 \right> \notag
\end{align}
By the third of Lemma~\ref{lemitomotegiteragaito}, we have $1=[b^{-q},a^{p}]\in \left< \left< [b^{-1},a]\right> \right>^{+}$ in $G_{p,-q}$. 
Thus $[b^{-1},a]$ is a generalized torsion element in $G_{p,-q}$ unless $[b^{-1},a]=1$ holds i.e $a$ and $b$ commute in $G_{p,-q}$. 
In the following, we suppose that $a$ and $b$ commute in $G_{p,-q}$ and prove that $a$ (and $b$) are generalized torsion elements in $G_{p,-q}$ under this assumption. 

Under the assumption, we have $ta^{p}t^{-1}=a^{p}b^{-1}$ and $ta^{-p}b^{pq}t^{-1}=t\left( a^{-1}b^{q}\right)^{p}t^{-1}=b^{pq}$. 
Thus ``informally'', the representation matrix $A$ of the action of $t(\cdot)t^{-1}$ on the subgroup of $G_{p,-q}$ generated by $a$ and $b$ under a basis $\{a,b\}$ is 
$A=\left( \begin{array}{cc}
     1   &  \frac{1}{q}  \\
      -\frac{1}{p}   &  1-\frac{1}{pq} \\
  \end{array} \right)$. 
Note that since we have $a^{p}=t^{-1}b^{-1}a^{p}t$ and $b^{q}a^{-1}=t^{-1}b^{q}t$ from the presentation, the representation matrix of the action of $t^{-1}(\cdot)t$ on this group under the same basis is 
$A^{-1}=\left( \begin{array}{cc}
     1 -\frac{1}{pq}   &  -\frac{1}{q}  \\
      \frac{1}{p}   &  1\\
  \end{array} \right)$. 
The word ``informally'' is used because the result of the action $t(\cdot)t^{-1}$ is not necessary closed in the subgroup generated by $a$ and $b$. 
However, for an element of the subgroup whose exponents of both of $a$ and $b$ are multiples of $pq$, the result of the action $t(\cdot)t^{-1}$ is in the subgroup: 
For $a^{(pq)s}b^{(pq)l}$ where $s$ and $l$ are integers, we have 
\begin{align}
t\left( a^{(pq)s}b^{(pq)l}\right)t^{-1}&=t\left( a^{pqs+pl}\right)t^{-1}\cdot t\left( a^{-pl}b^{pql}\right)t^{-1} \notag \\
     &= t\left( a^{p} \right)^{qs+l}t^{-1}\cdot t\left( a^{-1}b^{q} \right)^{pl}t^{-1} \notag \\
     &= \left(a^{p}b^{-1} \right)^{qs+l}\cdot \left( b^{q} \right)^{pl} \notag \\
     &= a^{pqs+pl}b^{-qs+(pq-1)l} \notag
\end{align}
This is represented as 
$A
\begin{pmatrix}
 (pq)s\\(pq)l
\end{pmatrix}
=\left( \begin{array}{cc}
     1   &  \frac{1}{q}  \\
      -\frac{1}{p}   &  1-\frac{1}{pq} \\
  \end{array} \right)\begin{pmatrix}
 (pq)s\\(pq)l
\end{pmatrix}
=\begin{pmatrix}
 pqs+pl\\-qs+(pq-1)l
\end{pmatrix}$. 
Similarly, for non-negative integer $k$, the exponents of $a$ and $b$ in $t^{k}\left( a^{(pq)^{k}s}b^{(pq)^{k}l}\right)t^{-k}$ are obtained from $A^{k}
\begin{pmatrix}
 (pq)^{k}s\\(pq)^{k}l
\end{pmatrix}$ 
and those of $t^{-k}\left( a^{(pq)^{k}s}b^{(pq)^{k}l}\right)t^{k}$ are obtained from $A^{-k}
\begin{pmatrix}
 (pq)^{k}s\\(pq)^{k}l
\end{pmatrix}$ .

We use the following Lemma: 

\begin{lem}\label{matrix}
Let $M\in GL_{2}(\mathbb{R})$ be a matrix such that one of whose eigenvalues $\lambda$ is in $\mathbb{C}\setminus \{1\}$, has length $\left| \lambda \right|$ one and the real part ${\rm Re}(\lambda)$ of $\lambda$ is in $\mathbb{Q}$. 
Then there exist positive integers $k$, $n$ and $m$ such that $nM^{k}+mI_{2}+nM^{-k}=O_{2}$ holds, where $I_2$ and $O_2$ are the $2\times2$ identity and zero matrices, respectively. 
\end{lem}
\begin{proof}
We regard $M$ as in $GL_{2}(\mathbb{C})$ through the natural inclusion. 
Since the characteristic polynomial of $M$ has real coefficients, the other eigenvalue is the conjugate of $\lambda$ and this is $\lambda^{-1}$ since $|\lambda|=1$. 
Let $\mathcal{V}$ and $\mathcal{V}'$ be eigenvectors of $\lambda$ and $\lambda^{-1}$, respectively. 
Note that $\mathcal{V}$ and $\mathcal{V}'$ are linearly independent. 
Since $\lambda \neq 1$ and $\left| \lambda \right|=1$, there is some positive integer $k$ such that ${\rm Re}(\lambda^{k})<0$. 
Note that ${\rm Re}(\lambda^{-k})={\rm Re}(\lambda^{k})$. 
Moreover, since $\left| \lambda \right|=1$, ${\rm Re}(\lambda^{k})=T_{k}\left({\rm Re}(\lambda)\right)$ holds, where $T_{k}(\cdot)$ is the $k$-th Chebyshev polynomial. 
We see that ${\rm Re}(\lambda^{k})$ is in $\mathbb{Q}$ since ${\rm Re}(\lambda)$ also is in and $T_{k}$ has integer coefficients. 
Thus there exist positive integers $m$ and $n$ such that $\frac{m}{n}=-2{\rm Re}(\lambda^{k})$. 
Then 
\begin{align}
\left(nM^{k}+nM^{-k}\right)\mathcal{V}=\left( n\lambda^{k}+n\lambda^{-k} \right)\mathcal{V}=2n{\rm Re}(\lambda^k)\mathcal{V}=-m\mathcal{V}=-mI_{2}\mathcal{V}. \notag
\end{align}
Similarly, the equation obtained by replacing $\mathcal{V}$ in the above equation with $\mathcal{V}'$ also holds with the same $k$, $m$ and $n$. 
Since $\mathcal{V}$ and $\mathcal{V}'$ span $\mathbb{C}^2$, the equation $nM^{k}+mI_{2}+nM^{-k}=O_{2}$ holds. 
\end{proof}

One of the eigenvalues of $A$ is $\frac{2pq-1}{2pq}+\frac{\sqrt{4pq-1}}{2pq}\sqrt{-1}$, which is a complex number of length one and whose real part is rational. 
Hence we can apply Lemma~\ref{matrix} and we get positive integers $k$, $n$ and $m$ such that $nA^{k}+mI_{2}+nA^{-k}=O_{2}$ holds. 
Using this equation, we have:
\begin{align}
\left(nA^{k}+mI_{2}+nA^{-k}\right) 
\begin{pmatrix}
 (pq)^{k}\\0
\end{pmatrix}=
\begin{pmatrix}
 0\\0
\end{pmatrix}
\Longleftrightarrow \left(t^{k}a^{(pq)^{k}}t^{-k}\right)^{n}\cdot \left(a^{(pq)^k}\right)^{m} \cdot \left(t^{-k}a^{(pq)^{k}}t^{k}\right)^{n}=1. \notag
\end{align}

This implies that $1\in \left<\left< a\right>\right>^{+}$. 
This $a$ is actually a generalized torsion element. 
For a contradiction, suppose that $a=1$ holds in $G_{p,-q}$. 
Then we have $G_{p,-q}\cong \mathbb{Z}$ from the presentation at the beginning of this section. 
However, $G_{p,-q}=\pi_{1}\left( K_{p,-q}(0)\right)$ has $\mathbb{Z}^{2}$ as a subgroup since $K_{p,-q}(0)$ has an incompressible torus, which is obtained by capping off a Seifert surface of minimal genus. 
This leads a contradiction. 
Therefore, $a$ is a generalized torsion element in $G_{p,-q}$ under the assumption of the commutativity of $a$ and $b$. 
Note that by the same argument, we see that $b$ is also a generalized torsion element in $G_{p,-q}$ under the assumption.

As a result, at least one of $a$ and $[b^{-1},a]$ is a generalized torsion element in $G_{p,-q}=\pi_{1}\left( K_{p,-q}(0)\right)$. 
This finishes a proof of Theorem~\ref{mainthm2}.

\section{A proof of Theorem~\ref{mainthm3}} \label{secprfofthm3}
For negative $pq$, the group $G_{p,q}$ admits a generalized torsion element, and thus is not bi-orderable as in Theorem~\ref{mainthm2}. 
In the following, we assume $p$ and $q$ are positive and $G_{p,q}$ admits no generalized torsion elements. 
We will give a bi-order. 
The following transformation of a presentation of $\pi_{1}(E_{p,q})$ indicates the isomorphisms among $\pi_{1}(E_{p,q})$, $\pi_{1}(E_{-q,-p})$ and $\pi_{1}(E_{q,p})$. 
We give some of the presentations names for readability, although some of these already are given names before. 
This transformation holds even if we do not impose the positivity of $p$ and $q$:
\begin{align}
  \pi_{1}(E_{p,q})&= \left<x, y \mid \{ (yx)^{q}(xy)^{-q})\}^{p}  y=x^{-1}\{ (yx)^{q}(xy)^{-q})\}^{p} \right> \label{bio1} \\
   &=  \left<x, y,b,t \mid \{ (yx)^{q}(xy)^{-q})\}^{p}  y=x^{-1}\{ (yx)^{q}(xy)^{-q})\}^{p},\ b=xy,\  t=y^{-1} \right> \notag \\
   &= \left<b, t \mid t(t^{-1}b^{q}tb^{-q})^{p}t^{-1}=b^{-1}(t^{-1}b^{q}tb^{-q})^{p} \right> \notag \\
    &= \left<a, b, t \mid t(t^{-1}b^{q}tb^{-q})^{p}t^{-1}=b^{-1}(t^{-1}b^{q}tb^{-q})^{p},\ a=t^{-1}b^{q}tb^{-q} \right> \notag \\
    &= \left<a, b, t \mid ta^{p}t^{-1}=b^{-1}a^{p}, \ tb^{-q}a^{-1}t^{-1}=b^{-q}  \right> \label{bio2} \\
    &= \left<a, b, t, A, B, T \mid ta^{p}t^{-1}=b^{-1}a^{p}, \ tb^{-q}a^{-1}t^{-1}=b^{-q},\ A=b^{-1}, \ B=a^{-1},\ T=t^{-1}  \right> \notag  \\
    &= \left<A, B, T \mid T^{-1}B^{-p}T=AB^{-p},\ T^{-1}A^{q}BT=A^{q}  \right> \notag \\
    &= \left<A, B, T \mid TA^{(-q)}T^{-1}=B^{-1}A^{(-q)},\ TB^{-(-p)}A^{-1}T^{-1}=B^{-(-p)}  \right> \label{bio3} \\
     &= \left< B, T \mid \{T^{-1}B^{(-p)}TB^{-(-p)}\}^{-q}T^{-1}=T^{-1}B^{-1} \{T^{-1}B^{(-p)}TB^{-(-p)}\}^{-q} \right> \notag \\
     &= \left< B, T, X, Y \mid \{T^{-1}B^{(-p)}TB^{-(-p)}\}^{-q}T^{-1}=T^{-1}B^{-1} \{T^{-1}B^{(-p)}TB^{-(-p)}\}^{-q},\ X=BT,\ Y=T^{-1} \right> \notag \\
     &= \left< X, Y \mid \{(YX)^{-p}(XY)^{-(-p)}\}^{-q}Y=X^{-1}\{(YX)^{-p}(XY)^{-(-p)}\}^{-q} \right> \label{bio4} \\
     &= \left< X, Y \mid Y\{(YX)^{-p}(XY)^{p}\}^{q}=\{(YX)^{-p}(XY)^{p}\}^{q}X^{-1} \right> \notag \\
     &= \left< X, Y, \mathbb{X}, \mathbb{Y} \mid Y\{(YX)^{-p}(XY)^{p}\}^{q}=\{(YX)^{-p}(XY)^{p}\}^{q}X^{-1} ,\ \mathbb{X}=Y^{-1},\ \mathbb{Y}=X^{-1} \right> \notag \\
     &= \left<\mathbb{X}, \mathbb{Y} \mid \{(\mathbb{YX})^{p}(\mathbb{XY})^{-p}\}^{q}\mathbb{Y}=\mathbb{X}^{-1}\{(\mathbb{YX})^{p}(\mathbb{XY})^{-p}\}^{q} \right> \label{bio5} \\
     &= \left<\mathbb{X}, \mathbb{Y}, \mathbb{B}, \mathbb{T} \mid \{(\mathbb{YX})^{p}(\mathbb{XY})^{-p}\}^{q}\mathbb{Y}=\mathbb{X}^{-1}\{(\mathbb{YX})^{p}(\mathbb{XY})^{-p}\}^{q},\ \mathbb{B}=\mathbb{XY},\ \mathbb{T}=\mathbb{Y}^{-1} \right> \notag \\
   &= \left<\mathbb{B}, \mathbb{T} \mid \mathbb{T} \left( \mathbb{T}^{-1}\mathbb{B}^{p}\mathbb{T}\mathbb{B}^{-p}\right)^{q}\mathbb{T}^{-1}=\mathbb{B}^{-1}\left(\mathbb{T}^{-1}\mathbb{B}^{p}\mathbb{T}\mathbb{B}^{-p}\right)^{q} \right> \notag \\
   &= \left<\mathbb{A},\mathbb{B}, \mathbb{T} \mid \mathbb{T} \left( \mathbb{T}^{-1}\mathbb{B}^{p}\mathbb{T}\mathbb{B}^{-p}\right)^{q}\mathbb{T}^{-1}=\mathbb{B}^{-1}\left(\mathbb{T}^{-1}\mathbb{B}^{p}\mathbb{T}\mathbb{B}^{-p}\right)^{q} , \ \mathbb{A}=\mathbb{T}^{-1}\mathbb{B}^{p}\mathbb{T}\mathbb{B}^{-p} \right> \notag \\
    &= \left<\mathbb{A}, \mathbb{B}, \mathbb{T} \mid \mathbb{T}\mathbb{A}^{q}\mathbb{T}^{-1}=\mathbb{B}^{-1}\mathbb{A}^{q}, \ \mathbb{T}\mathbb{B}^{-p}\mathbb{A}^{-1}\mathbb{T}^{-1}=\mathbb{B}^{-p}  \right>  \label{bio6}
\end{align}  

Tracing the above transformation, we see that $\mathbb{B}$ and $\mathbb{T}$ in \eqref{bio6} correspond to $a$ and $a^{-1}t^{-1}$ in \eqref{bio2} respectively, and $b$ in \eqref{bio2} corresponds to $\mathbb{A}^{-1}$ in \eqref{bio6}. Thus $[\mathbb{B}^{p},\mathbb{A}^{q}]$ in \eqref{bio6} corresponds to $[a^{p},b^{-q}]=b^{q}[b^{q},a^{p}]b^{-q}$ in \eqref{bio2}. 
This implies that under the $0$-surgery, the relations $[b^{q},a^{p}]=1$ and $[\mathbb{B}^{p},\mathbb{A}^{q}]=1$ are equivalent. 
We pick up presentations of $G_{p,q}=\pi_{1}\left(K_{p,q}(0)\right)$ we needed. 

\begin{align}
  G_{p,q}&= \left<a, b, t \mid ta^{p}t^{-1}=b^{-1}a^{p}, \ tb^{-q}a^{-1}t^{-1}=b^{-q} ,\ [b^{q},a^{p}]=1  \right> \label{biosurg1} \\
    &= \left<x, y \mid \{ (yx)^{q}(xy)^{-q})\}^{p}  y=x^{-1}\{ (yx)^{q}(xy)^{-q})\}^{p},\ \{ (yx)^{-q}(xy)^{q})\}^{p}\{ (yx)^{q}(xy)^{-q})\}^{p}=1 \right> \label{biosurg2} \\
     &= \left<\mathbb{X}, \mathbb{Y} \mid \{(\mathbb{YX})^{p}(\mathbb{XY})^{-p}\}^{q}\mathbb{Y}=\mathbb{X}^{-1}\{(\mathbb{YX})^{p}(\mathbb{XY})^{-p}\}^{q},\ \{ (\mathbb{YX})^{-p}(\mathbb{XY})^{p})\}^{q}\{ (\mathbb{YX})^{p}(\mathbb{XY})^{-p})\}^{q}=1 \right> \label{biosurg3} \\
    &= \left<\mathbb{A}, \mathbb{B}, \mathbb{T} \mid \mathbb{T}\mathbb{A}^{q}\mathbb{T}^{-1}=\mathbb{B}^{-1}\mathbb{A}^{q}, \ \mathbb{T}\mathbb{B}^{-p}\mathbb{A}^{-1}\mathbb{T}^{-1}=\mathbb{B}^{-p} ,\ [\mathbb{B}^{p},\mathbb{A}^{q}]=1 \right>  \label{biosurg4}
\end{align} 

By Lemma~\ref{lemcandidate} and the assumption of admitting no generalized torsion elements, we have $[xy,yx]=1$ holds in \eqref{biosurg2}, which corresponds to $[b,t^{-1}bt]=1$ in \eqref{biosurg1}. 
Similarly, $[\mathbb{XY},\mathbb{YX}]=1$ holds in \eqref{biosurg3}, which corresponds to $[\mathbb{B},\mathbb{T}^{-1}\mathbb{BT}]=1$ in \eqref{biosurg4}. 

\begin{prop} \label{commute}
For every integer $n$, the relations $[b, t^{-n}bt^{n}]=1$ and $[\mathbb{B},\mathbb{T}^{-n}\mathbb{B}\mathbb{T}^{n}]=1$ hold in \eqref{biosurg1} and \eqref{biosurg4}, respectively, under the assumption of admitting no generalized torsion elements and the positivity of $p$ and $q$.  
\end{prop}
\begin{proof}
Note that since $t^{n}[b,t^{-n}bt^{n}]t^{-n}=[t^{n}bt^{-n},b]$, it is enough to prove for non-negative $n$. 
We prove by induction on $n$. 
The case $n=0$ is trivial and the case $n=1$ is stated above. \\
Let $n$ is a positive integer and assume that $[b,t^{-k}bt^{-k}]=1$ and $[\mathbb{B},\mathbb{T}^{-k}\mathbb{B}\mathbb{T}^{k}]=1$ hold in \eqref{biosurg1} and \eqref{biosurg4}, respectively, for every integer $k$ such that $|k|\leq n$. 
The element $[\mathbb{B}^{-1},\mathbb{T}^{-n}\mathbb{B}^{-1}\mathbb{T}^{n}]=[\mathbb{B},\mathbb{T}^{n}]^{-1}\cdot[\mathbb{B},\mathbb{T}^{-n}\mathbb{B}\mathbb{T}^{n}]\cdot[\mathbb{B},\mathbb{T}^{n}]$ is represented as follows in \eqref{biosurg1}: 
\begin{align}
1&=[\mathbb{B}^{-1},\mathbb{T}^{-n}\mathbb{B}^{-1}\mathbb{T}^{n}] \notag \\
  &=[a^{-1},(ta)^{n}a^{-1}(ta)^{-n}] \notag \\
  &=[b^{q}t^{-1}b^{-q}t, (b^{q}tb^{-q})^{n}\cdot b^{q}t^{-1}b^{-q}t \cdot (b^{q}tb^{-q})^{-n}] \notag \\
  &=[b^{q}\cdot t^{-1}b^{-q}t, (b^{q}t^{n}b^{-q})\cdot b^{q}t^{-1}b^{-q}t \cdot (b^{q}t^{-n}b^{-q})] \notag \\
  &=\left[ [b^{-1},t]^{q}, b^{q}t^{n-1}b^{-q}t b^{q}t^{-n}b^{-q} \right] \notag \\
  &=\left[ [b^{-1},t]^{q}, b^{q}\cdot t^{n-1}b^{-q}t^{-(n-1)}\cdot t^{n} b^{q}t^{-n}\cdot b^{-q} \right] \notag \\
  &=\left[ [b^{-1},t]^{q},  t^{n-1}b^{-q}t^{-(n-1)}\cdot t^{n} b^{q}t^{-n} \right] \notag \\
  &=\left[ [b^{-1},t]^{q},  t^{n-1}\cdot b^{-q} \cdot t b^{q} t^{-1}\cdot t^{-(n-1)} \right] \notag \\
  &=\left[ [b^{-1},t]^{q},  \left( t^{n-1}\cdot [b,t^{-1}]\cdot t^{-(n-1)}\right)^{q} \right] \notag 
\end{align}
Note that $[b,t^{-k}bt^{k}]=1$ implies $[b^{-1},t^{-k}bt^{k}]=1$ and so on. 
By the third of Lemma~\ref{lemitomotegiteragaito}, we have $1 \in \left< \left< \left[ [b^{-1},t],   t^{n-1}\cdot [b,t^{-1}]\cdot t^{-(n-1)} \right] \right> \right>^{+} $. 
Since $G_{p,q}$ admits no generalized torsion elements, we have $\left[ [b^{-1},t],  t^{n-1}\cdot [b,t^{-1}]\cdot t^{-(n-1)} \right]=1$. 
We compute as follows: 
\begin{align}
1&=\left[ [b^{-1},t],  t^{n-1}\cdot [b,t^{-1}]\cdot t^{-(n-1)} \right] \notag \\
  &=\left(t^{-1}btb^{-1}\right) \left(t^{n}b^{-1}t^{-1}bt^{-(n-1)} \right) \left(bt^{-1}b^{-1}t\right) \left( t^{n-1}b^{-1}tbt^{-n} \right) \notag \\
  &= (t^{-1}bt)b^{-1} (t^{n}b^{-1}t^{-n}) (t^{n-1}bt^{-(n-1)})b(t^{-1}b^{-1}t) (t^{n-1}b^{-1}t^{-(n-1)}) (t^{n}bt^{-n}) \notag \\
  &= (t^{-1}bt) (t^{n}b^{-1}t^{-n}) (t^{n-1}bt^{-(n-1)})(t^{-1}b^{-1}t) (t^{n-1}b^{-1}t^{-(n-1)}) (t^{n}bt^{-n}) \notag \\
  &= (t^{-1}bt) (t^{n}b^{-1}t^{-n})\cdot t^{-1} (t^{n}bt^{-n})(b^{-1})t\cdot (t^{n-1}b^{-1}t^{-(n-1)}) (t^{n}bt^{-n}) \notag \\
  &= (t^{-1}bt) (t^{n}b^{-1}t^{-n})\cdot t^{-1} (b^{-1})(t^{n}bt^{-n})t\cdot (t^{n-1}b^{-1}t^{-(n-1)}) (t^{n}bt^{-n}) \notag \\
  &= (t^{-1}bt) (t^{n}b^{-1}t^{-n})\cdot (t^{-1} b^{-1}t)(t^{n-1}bt^{-(n-1)})\cdot (t^{n-1}b^{-1}t^{-(n-1)}) (t^{n}bt^{-n}) \notag \\
  &= (t^{-1}bt) (t^{n}b^{-1}t^{-n})(t^{-1} b^{-1}t) (t^{n}bt^{-n}) \notag \\
  &= t^{-1}bt^{n+1}b^{-1}t^{-(n+1)} b^{-1} t^{n+1}bt^{-n} \notag \\
  &=t^{-1}[b^{-1}, t^{n+1}bt^{-(n+1)}]t \notag
\end{align}
This implies $[b,t^{-(n+1)}bt^{n+1}]=1$. \\

To get $[\mathbb{B}, \mathbb{T}^{-(n+1)}\mathbb{B}\mathbb{T}^{n+1}]=1$, we introduce the following claim:
\begin{claim}
Under the assumption that $[\mathbb{B}, \mathbb{T}^{-k}\mathbb{B}\mathbb{T}^{k}]=1$ holds for any integer $k$ such that $|k|\leq n$, 
the equation $\mathbb{B}\mathbb{T}^{-1}\mathbb{B}^{-1}\mathbb{T}\cdot \left( \mathbb{BT}\right)^{m}=\left( \mathbb{BT}\right)^{m} \cdot \mathbb{T}^{-m}\mathbb{B}\mathbb{T}^{-1}\mathbb{B}^{-1}\mathbb{T}^{m+1}$ holds for any non-negative integer $m$ such that $m\leq n$. 
\end{claim}
\begin{proof}
We prove by induction. 
The case $m=0$ is trivial. 
Assume that the statement holds for some $m<n$. 
Then
\begin{align}
\mathbb{B}\mathbb{T}^{-1}\mathbb{B}^{-1}\mathbb{T}\cdot \left( \mathbb{BT}\right)^{m+1}&=\left( \mathbb{BT}\right)^{m} \cdot \mathbb{T}^{-m}\mathbb{B}\mathbb{T}^{-1}\mathbb{B}^{-1}\mathbb{T}^{m+1}\cdot \left( \mathbb{BT}\right) \notag \\
&=\left( \mathbb{BT}\right)^{m} \cdot \left( \mathbb{T}^{-m}\mathbb{B}\mathbb{T}^{m}\right)\cdot \left( \mathbb{T}^{-(m+1)}\mathbb{B}^{-1}\mathbb{T}^{m+1}\right)\cdot \mathbb{B}\cdot \mathbb{T} \notag \\
&=\left( \mathbb{BT}\right)^{m}\cdot \mathbb{B} \cdot \left( \mathbb{T}^{-m}\mathbb{B}\mathbb{T}^{m}\right)\cdot \left( \mathbb{T}^{-(m+1)}\mathbb{B}^{-1}\mathbb{T}^{m+1}\right)\cdot  \mathbb{T} \notag \\
&=\left( \mathbb{BT}\right)^{m}\cdot \mathbb{BT} \cdot \left( \mathbb{T}^{-(m+1)}\mathbb{B}\mathbb{T}^{m+1}\right)\cdot \left( \mathbb{T}^{-(m+2)}\mathbb{B}^{-1}\mathbb{T}^{m+2}\right) \notag \\
&=\left( \mathbb{BT}\right)^{m+1} \cdot \mathbb{T}^{-(m+1)}\mathbb{B}\mathbb{T}^{-1}\mathbb{B}^{-1}\mathbb{T}^{m+2} \notag
\end{align}
This finishes the induction.
\end{proof}

Using the claim above, we represent $1=[b,t^{n}bt^{-n}]$ in terms of $\{ \mathbb{B}, \mathbb{T} \}$ as follows:
\begin{align}
1&=[b,t^{n}bt^{-n}] \notag \\
&= \left[ \mathbb{A}^{-1}, \left( \mathbb{BT}\right)^{-n} \mathbb{A}^{-1} \left( \mathbb{BT}\right)^{n}\right] \notag \\
&= \left[ \mathbb{B}^{p}\mathbb{T}^{-1}\mathbb{B}^{-p}\mathbb{T}, \left( \mathbb{BT}\right)^{-n}\cdot \mathbb{B}^{p}\mathbb{T}^{-1}\mathbb{B}^{-p}\mathbb{T} \cdot  \left( \mathbb{BT}\right)^{n}\right] \notag \\
&= \left[ \left( \mathbb{B}\mathbb{T}^{-1}\mathbb{B}^{-1}\mathbb{T} \right)^{p}, \left( \mathbb{BT}\right)^{-n}\cdot \left( \mathbb{B}\mathbb{T}^{-1}\mathbb{B}^{-1}\mathbb{T}\right)^{p} \cdot  \left( \mathbb{BT}\right)^{n}\right] \notag \\
&= \left[ \left( \mathbb{B}\mathbb{T}^{-1}\mathbb{B}^{-1}\mathbb{T} \right)^{p},  \left(\mathbb{T}^{-n} \mathbb{B}\mathbb{T}^{-1}\mathbb{B}^{-1}\mathbb{T}^{n+1}\right)^{p} \right] \notag 
\end{align}
By the third of Lemma~\ref{lemitomotegiteragaito}, we have $1 \in \left< \left< \left[  \mathbb{B}\mathbb{T}^{-1}\mathbb{B}^{-1}\mathbb{T} ,  \mathbb{T}^{-n} \mathbb{B}\mathbb{T}^{-1}\mathbb{B}^{-1}\mathbb{T}^{n+1}\right] \right> \right>^{+} $. 
Since $G_{p,q}$ admits no generalized torsion elements, we have $\left[  \mathbb{B}\mathbb{T}^{-1}\mathbb{B}^{-1}\mathbb{T} ,  \mathbb{T}^{-n} \mathbb{B}\mathbb{T}^{-1}\mathbb{B}^{-1}\mathbb{T}^{n+1}\right]=1$. 
We compute as follows:
\begin{align}
1&=\left[  \mathbb{B}\mathbb{T}^{-1}\mathbb{B}^{-1}\mathbb{T} ,  \mathbb{T}^{-n} \mathbb{B}\mathbb{T}^{-1}\mathbb{B}^{-1}\mathbb{T}^{n+1}\right] \notag \\
&=\mathbb{T}^{-1}\mathbb{B}\mathbb{T}\mathbb{B}^{-1} \cdot \mathbb{T}^{-(n+1)}\mathbb{B}\mathbb{T}\mathbb{B}^{-1}\mathbb{T}^{n} \cdot \mathbb{B}\mathbb{T}^{-1}\mathbb{B}^{-1}\mathbb{T} \cdot  \mathbb{T}^{-n} \mathbb{B}\mathbb{T}^{-1}\mathbb{B}^{-1}\mathbb{T}^{n+1} \notag \\
&=\left( \mathbb{T}^{-1}\mathbb{B}\mathbb{T}\right) \cdot \mathbb{B}^{-1} \cdot \left( \mathbb{T}^{-(n+1)}\mathbb{B}\mathbb{T}^{n+1}\right) \cdot \left( \mathbb{T}^{-n}\mathbb{B}^{-1}\mathbb{T}^{n}\right) \cdot \mathbb{B} \cdot \left( \mathbb{T}^{-1}\mathbb{B}^{-1}\mathbb{T}\right) \cdot \left( \mathbb{T}^{-n} \mathbb{B}\mathbb{T}^{n}\right) \cdot \left( \mathbb{T}^{-(n+1)}\mathbb{B}^{-1}\mathbb{T}^{n+1}\right) \notag \\
&= \mathbb{B}^{-1} \cdot \left( \mathbb{T}^{-(n+1)}\mathbb{B}\mathbb{T}^{n+1}\right) \cdot  \mathbb{B}  \cdot \left( \mathbb{T}^{-(n+1)}\mathbb{B}^{-1}\mathbb{T}^{n+1}\right) \notag \\
&= \left[ \mathbb{B}  , \mathbb{T}^{-(n+1)}\mathbb{B}^{-1}\mathbb{T}^{n+1}\right]\notag 
\end{align}
In the above, we use the commutativity of $\mathbb{T}^{-1}\mathbb{BT}$ and $\mathbb{T}^{-(k+1)}\mathbb{BT}^{k+1}$, which is induced by that of $\mathbb{B}$ and $\mathbb{T}^{-k}\mathbb{BT}^{k}$ for integer $k$ such that $|k|\leq n$. 
The above computation implies $ \left[ \mathbb{B}  , \mathbb{T}^{-(n+1)}\mathbb{B}\mathbb{T}^{n+1}\right]=1$, and this finishes the induction. 
\end{proof}

By Proposition~\ref{commute}, we have $[t^{-m}bt^{m}, t^{-n}bt^{n}]=t^{-m}[b,t^{-(n-m)}bt^{n-m}]t^{m}=1$ for all integers $m$ and $n$. 
Moreover, $[t^{-m}at^{m},t^{-n}bt^{n}]=1$ and $[t^{-m}at^{m},t^{-n}at^{n}]=1$ hold for all integers $m$ and $n$ since $a=
t^{-1}b^{q}tb^{-q}=(t^{-1}bt)^{q}\cdot b^{-q}$. 
We give a lemma which will be used later: 

\begin{lem} \label{characteristic}
Under the assumption of admitting no generalized torsion elements and the positivity of $p$ and $q$, we have
$\left( tb^{pq}t^{-1}\right) \cdot b^{-(2pq+1)} \cdot \left( t^{-1}b^{pq}t\right)=1$ in $G_{p,q}$. 
\end{lem}
\begin{proof}
Under the assumption, we can use the commutativity in Proposition~\ref{commute}. 
Using the presentation \eqref{biosurg1} of $G_{p,q}$, we have:
\begin{itemize}
\item $tb^{pq}t^{-1}= \left(tb^{pq}a^{p}t^{-1}\right)\cdot \left( t a^{-p}t^{-1}\right)=\left( t b^{-q}a^{-1} t^{-1} \right)^{-p}\cdot \left( t a^{p}t^{-1}\right)^{-1}=\left(b^{-q}\right)^{-p}\cdot(b^{-1}a^{p})^{-1}=a^{-p}b^{pq+1}$, 
\item $t^{-1}b^{pq}t= \left( t^{-1} b^{-q} t \right)^{-p}=\left( b^{-q} a^{-1}\right)^{-p}=a^{p}b^{pq}$.
\end{itemize}
Using the commutativity, we get the statement. 
\end{proof}

\subsection*{Injecting $G_{p,q}$ with the assumption into some group $K$}
We assume that $p$ and $q$ are positive, and $G_{p,q}$ admits no generalized torsion elements. 
We will construct some group $K$ and show that $G_{p,q}$ with our assumption injects into this $K$, which will be shown to be bi-orderable later.

Consider a group $H$ with the following presentation. 
\begin{align}
H=\left<\tau,\  x_{n} \ (n\in \mathbb{Z}) \mid [x_{i},x_{j}]=1,\ x^{pq}_{i+1}\cdot x^{-(2pq+1)}_{i}\cdot x^{pq}_{i-1}=1,\ \tau x_{i} \tau^{-1}=x_{i+1} \ (i,j\in \mathbb{Z})\right> \label{grph}
\end{align}
\begin{lem}
Under the assumption, $G_{p,q}$ is isomorphic to $H$, where $\tau$ and $x_{n}$ in $H$ correspond to $t$ and $t^{n}bt^{-n}$ in $G_{p,q}$, respectively. 
\end{lem}
\begin{proof}
Under the assumption, a map from $H$ to $G_{p,q}$ sending $\tau$ and $x_n$ to $t$ and $t^{n}bt^{-n}$ sends every relation in \eqref{grph} to the identity element in $G_{p,q}$ by Proposition~\ref{commute} and Lemme~\ref{characteristic}, and this is a group homomorphism. 
For the other side, a map from $G_{p,q}$ to $H$ sending $t$ and $b$ to $\tau$ and $x_{0}$ sends every relation in \eqref{biosurg1} to the identity element in $H$ by noting that $a=t^{-1}b^{q}tb^{-q}$, the commutativity in Proposition~\ref{commute} and the relation in Lemma~\ref{characteristic} under the assumption, and this is also a group homomorphism. 
For the constructed two maps, one is the inverse of the other. 
Therefore $G_{p,q}$ and $H$ are isomorphic under the identification. 
\end{proof}

We identify $H$ with $G_{p,q}$. 
Take the abelianization $\rho: G_{p,q}\longrightarrow \mathbb{Z}$ such that $\rho(x_n)=0$ for every $n\in \mathbb{Z}$ and $\rho(\tau)=1$. 
Then we have the following: 
\begin{lem}
${\rm Ker}(\rho)$ has the following presentation, where $x_n$ is the same element of $H=G_{p,q}$ for every $n\in \mathbb{Z}$:
\begin{align}
{\rm Ker}(\rho)=\left<x_{n} \ (n\in \mathbb{Z}) \mid [x_{i},x_{j}]=1,\ x^{pq}_{i+1}\cdot x^{-(2pq+1)}_{i}\cdot x^{pq}_{i-1}=1\ (i,j\in \mathbb{Z})\right>  \label{kerh}
\end{align}
\end{lem}
\begin{proof}
It is easy to see that ${\rm Ker}(\rho)$ is generated by $x_{n}$ ($n\in \mathbb{Z}$) in $H=G_{p,q}$. 
Moreover, by using the relation of type $\tau x_{i} \tau^{-1}=x_{i+1}$, we see that every conjugate of every relation in \eqref{grph} in $H=G_{p,q}$ is equivalent to a conjugate of a relation (or its inverse) in \eqref{kerh}. 
For example, $\tau x_{k}[x_{i},x_{j}]x^{-1}_{k}\tau^{-1}$ is equivalent to $x_{k+1}[x_{i+1},x_{j+1}]x^{-1}_{k+1}$. 
This gives the presentation. 
\end{proof}

\begin{rmk}
The presentation \eqref{kerh} for ${\rm Ker}(\rho)$ is also obtained by noting that $H=G_{p,q}$ has a structure of the HNN-extension of a group. 
\end{rmk}

\begin{defini}
Take a matrix 
$A=\left( \begin{array}{cc}
     1   &  -\frac{1}{q}  \\
      -\frac{1}{p}   &  1+\frac{1}{pq} \\
  \end{array} \right)\in GL_{2}(\mathbb{R})$. \\
Define $K= \mathbb{R}^{2}\rtimes \mathbb{Z}$ as a group whose product is defined by
\begin{align}
 \left(\begin{pmatrix}
 s_1\\l_1
\end{pmatrix}, n_1\right)\bullet
\left(\begin{pmatrix}
 s_2\\l_2
\end{pmatrix}, n_2\right)=
\left(\begin{pmatrix}
 s_1\\l_1
\end{pmatrix}+A^{n_1}
\begin{pmatrix}
 s_2\\l_2
\end{pmatrix}, n_{1}+n_2\right) \notag
\end{align} 
\end{defini}
We will show that $G_{p,q}$ under the assumption is a subgroup of $K$ after some preparations. 
For every generator $x_{n}$ in ${\rm Ker}(\rho)$ with a presentation \eqref{kerh}, we assign $A^{n}\begin{pmatrix}
 0\\1
\end{pmatrix}\in \mathbb{R}^{2}$, which is denoted by $r(x_{n})$. 
Extend this map to a homomorphism $r:{\rm Ker}(\rho)\longrightarrow \mathbb{R}^{2}$. 
An element $w=\prod \limits_{k=i}^{j}\left(x_{n_{k}}\right)^{m_{k}} \in {\rm Ker}(\rho)$ for some integers $i\leq j$, $n_{k}$'s and $m_{k}$'s is sent to $r(w)=\sum \limits_{k=i}^{j}\left(m_{k}A^{n_{k}}
\begin{pmatrix}
 0\\1
\end{pmatrix}\right) \in \mathbb{R}^2$. 
Note that the order of the products in $w$ is not important since $[x_{i},x_{j}]=1$ holds. 
Note also that $r(\cdot)$ is well-defined since $r\left( x^{pq}_{i+1}\cdot x^{-(2pq+1)}_{i}\cdot x^{pq}_{i-1}\right)=\left(pqA^{i+1}-(2pq+1)A^{i}+pqA^{i-1}\right)
\begin{pmatrix}
 0\\1
\end{pmatrix}=pq\left(A^{2}-(2+\frac{1}{pq})A+I_{2}\right)A^{i-1}
\begin{pmatrix}
 0\\1
\end{pmatrix}=\begin{pmatrix}
 0\\0
\end{pmatrix}$ by the Cayley-Hamilton's theorem. 

\begin{defini}
Let $\phi: G_{p,q}\longrightarrow K$ be a map defined by $\phi(g)=\left(r\left(g\cdot \tau^{-\rho(g)}\right) ,\rho(g) \right)$ for $g\in H= G_{p,q}$. 
Note that this $\phi$ is a homomorphism. 
\end{defini}

\begin{prop}\label{inject}
A map $\phi$ defined above is injective. 
\end{prop}
\begin{proof}
It is enough to show that $r:{\rm Ker}(\rho)\longrightarrow K$ is injective. 
Suppose that $r\left( \prod \limits_{k=i}^{j}\left(x_{n_{k}}\right)^{m_{k}}\right)=\sum \limits_{k=i}^{j}\left(m_{k}A^{n_{k}}
\begin{pmatrix}
 0\\1
\end{pmatrix}\right)=
\begin{pmatrix}
 0\\0
\end{pmatrix}$ holds. 
Take eigenvectors $V_{+}=\frac{1}{2\sqrt{4pq+1}}\begin{pmatrix}
 -2p\\1+\sqrt{4pq+1}
\end{pmatrix}$ and $V_{-}=\frac{1}{2\sqrt{4pq+1}}\begin{pmatrix}
 -2p\\1-\sqrt{4pq+1}
\end{pmatrix}$ for eigenvalues $\lambda_{+}=\frac{(2pq+1)\sqrt{4pq+1}}{2pq}$ and $\lambda_{-}=\frac{(2pq+1)-\sqrt{4pq+1}}{2pq}$, respectively for $A$. 
Note that $V_{+}$ and $V_{-}$ are linearly independent, and $\begin{pmatrix}
 0\\1
\end{pmatrix}=V_{+}-V_{-}$. 
Then: 
\begin{align}
\begin{pmatrix}
 0\\0
\end{pmatrix}&=\sum \limits_{k=i}^{j}\left(m_{k}A^{n_{k}}
\begin{pmatrix}
 0\\1
\end{pmatrix}\right) \notag  \\
 &=\sum \limits_{k=i}^{j}\left(m_{k}A^{n_{k}}
\left(V_{+}-V_{-}\right)\right) \notag \\
 &=\sum \limits_{k=i}^{j}\left(m_{k}
\left(\lambda^{n_{k}}_{+}V_{+}-\lambda^{n_{k}}_{-}V_{-}\right)\right) \notag \\
 &=\left( \sum \limits_{k=i}^{j} m_{k}
\lambda^{n_{k}}_{+}\right)V_{+}-\left( \sum \limits_{k=i}^{j} m_{k}
\lambda^{n_{k}}_{-}\right)V_{-} \notag
\end{align}
This implies that a Laurent polynomial $f(x)=\sum \limits_{k=i}^{j} m_{k}
x^{n_{k}}$ with integer coefficients has roots $\lambda_{+}$ and $\lambda_{-}$. 
Thus there exist a non-zero integer $N$ and a Laurent polynomial $g(x)$ with integer coefficients such that $f(x)=\frac{1}{N}\left( pqx^{2}-(2pq+1)+pqx^{-1}\right) g(x)$ holds. 
This implies that we can represent $f(x)$ as $f(x)=\sum \limits_{k=i}^{j} m_{k}
x^{n_{k}}=\frac{1}{N}\sum \limits_{k=i'}^{j'}M_{k}\left( pqx^{N_{k}+1}-(2pq+1)x^{N_{k}}+pqx^{N_{k}-1} \right)$ for some integers $i'\leq j'$, $N_{k}$'s and $M_{k}$'s. 
Then we can represent $\left( \prod \limits_{k=i}^{j}\left(x_{n_{k}}\right)^{m_{k}}\right)^{N}$ as:
\begin{align}
 \left( \prod \limits_{k=i}^{j}\left(x_{n_{k}}\right)^{m_{k}}\right)^{N} &=  \prod \limits_{k=i'}^{j'}\left(  \left(x_{N_{k}+1}\right)^{pq} \cdot \left( x_{N_{k}}\right)^{-(2pq+1)} \cdot \left( x_{N_{k}-1} \right)^{pq}   \right)^{M_{k}} =1 \notag 
\end{align}
Thus $\prod \limits_{k=i}^{j}\left(x_{n_{k}}\right)^{m_{k}}$ is a torsion element or the identity element in ${\rm Ker}(\rho)\triangleleft H=G_{p,q}$. 
In our assumption, there are no torsion elements in $G_{p,q}$. 
Hence $\prod \limits_{k=i}^{j}\left(x_{n_{k}}\right)^{m_{k}}=1$, and $r(\cdot)$ is injective.

\end{proof}

\subsection*{A bi-order on $K$ }
We construct a bi-order on $K$ by imitating the construction in \cite{perronrolfsen}. 
Since $G_{p,q}$ with our assumption injects into $K$ by Proposition~\ref{inject}, this bi-order makes $G_{p,q}$ with our assumption bi-orderable. 

Take eigenvectors $V_{\pm}$ and eigenvalues $\lambda_{\pm}$ of $A$ defined in the proof of Proposition~\ref{inject}. 
Note that $\lambda_{\pm}$ are real and positive. 
Since $V_{+}$ and $V_{-}$ are linearly independent, every element of $\mathbb{R}^{2}$ is uniquely represented as a linear combination of $V_{+}$ and $V_{-}$. 

\begin{defini}
Define an order $<$ on $K=\mathbb{R}^{2}\rtimes \mathbb{Z}$ as follows, where $a,b,c$ and $d$ are real numbers:\\
$\left( aV_{+}+bV_{-}, n\right)\ <\ \left( cV_{+}+dV_{-}, m\right)$ if
\begin{itemize}
\item $n<m$ as integers, 
\item $n=m$ and $a<c$ as real numbers, or
\item $n=m$, $a=c$ and $b<d$ as real numbers.
\end{itemize}
\end{defini}

It is easy to check that the order $<$ on $K$ defined above is invariant under the multiplying elements from left and right i.e $<$ is bi-order. 
This finishes a proof of Theorem~\ref{mainthm3}. 

\begin{rmk}~\label{candidates}
As we see in the proof of Theorem~\ref{mainthm3} in this section, 
if $G_{p,q}$ with positive integers $p$ and $q$ are not bi-orderable, then there is a torsion element in $G_{p,q}$ (as in the proof of Proposition~\ref{inject}) or there is a generalized torsion element $[b,t^{-n}bt^{n}]$ in \eqref{biosurg1} or $[\mathbb{B},\mathbb{T}^{-n}\mathbb{B}\mathbb{T}^{n}]$ in \eqref{biosurg4} for some integer $n$ (as in the proof of Proposition~\ref{commute}). 
Moreover, if $[b,t^{-n}bt^{n}]$ in \eqref{biosurg1} or $[\mathbb{B},\mathbb{T}^{-n}\mathbb{B}\mathbb{T}^{n}]$ in \eqref{biosurg4} is a non-trivial element for some integer $n$, then there exists an integer $m$ such that $|m|\leq |n|$ and $[b,t^{-m}bt^{m}]$ in \eqref{biosurg1} or $[\mathbb{B},\mathbb{T}^{-m}\mathbb{B}\mathbb{T}^{m}]$ in \eqref{biosurg4} is a generalized torsion element (as in the proof of Proposition~\ref{commute}).
\end{rmk}

\section{Some computations}\label{seccomputations}
Unfortunately, we cannot know whether $[xy,yx]$ is a generalized torsion element in $G_{p,q}$ under the presentation \eqref{anotherpresentation2} for each $(p,q)$ by Theorem~\ref{mainthm} only. 
In this section, we try to construct a homomorphism from $G_{p,q}$ to some symmetric group $S_{n+1}$ which maps $[xy,yx]$ to a non-trivial element by using a computer. 
Note that the element $[xy,yx]$ in the presentation \eqref{anotherpresentation2} corresponds to $[b,t^{-1}bt]$ in the presentation \eqref{stdrep}. 
As stated at Remark~\ref{candidates}, if one of $[b,t^{-n}bt^{n}]$ in the presentation \eqref{biosurg1} or $[\mathbb{B}, \mathbb{T}^{-n}\mathbb{B}\mathbb{T}^{n}]$ in the presentation \eqref{biosurg4} for some integer $n$ is a non-trivial element of $G_{p,q}$ for $pq>0$, we can conclude that $G_{p,q}$ admits a generalized torsion element. 
Since $G_{p,q}$ is a 3-manifold group i.e. the fundamental group of a compact 3-manifold, it is residually finite (\cite{hempel}, for example). 
Thus  the non-triviality of every non-trivial element of $G_{p,q}$ is detected by some homomorphism into some symmetric group. 
This detection is proposed to the author by Professor Nozaki. 
We list homomorphisms from $G_{p,q}$'s to some symmetric groups which map $[xy,yx]$'s to non-trivial elements in Table 1. 
Note that when $(p,q)=(1,1)$ or $(1,-1)$, the element $[xy,yx]$ is always the identity element in the presentation \eqref{anotherpresentation2}. 
It is known that $G_{1,1}$ admits no generalized torsion elements and that $G_{1,-1}$ admits a generalized torsion element. 
We impose on these homomorphisms the condition that $x$ is mapped to $[1,2,\dots,n,0]\in S_{n+1}$ and $n$ is less than ten. 
This imposition is only due to the limitations of the author's computer and his skill of programming. 
It may be true that except for $(p,q)=\pm(1,1)$, the group $G_{p,q}=\pi_{1}\left(K_{p,q}(0)\right)$ admits a generalized torsion element.

\subsection*{Notation for symmetric groups}
The symmetric group $S_{n+1}$ is a group consisting of bijections on a set $\{0,1,\dots, n\}$ for a non-negative integer $n$. 
By $[a_0,\dots,a_n]$, we represent a element of $S_{n+1}$ which maps $i$ to $a_{i}$. 
For two elements $a$ and $b$ of $S_{n+1}$, their product $ab$ is a bijection obtained by composing $a$ and $b$. 
In this paper, $b$ is carried out at first and $a$ is carried out at next in $ab$. 
For example, $[2,0,1]\cdot[1,0,2]=[0,2,1]$.

\begin{table}
\begin{center}
 \caption{A homomorphisms from $\pi_{1}\left(K_{p,q}(0)\right)$ to some symmetric group which sends $[xy,yx]$ to a non-trivial element}
\begin{tabular}{|c|c|c|c|} \hline  \label{tableqplus}
 $(p,q)$ & the symmetric group $S_{n+1}$ & the image of $x$ & the image of $y$ \\ \hline \hline
% $(1,1)$ &     $\times$  &   $ \times$   &    $\times $   \\ \hline
 $(2,1)$ & $S_{8}$ & $[1,2,3,4,5,6,7,0]$ & $[5,2,4,0,6,1,7,3]$ \\ \hline
 $(3,1)$ & $S_{7}$ & $[1,2,3,4,5,6,0]$ & $[3,0,5,2,6,4,1]$ \\ \hline
 $(4,1)$ & $S_{7}$ & $[1,2,3,4,5,6,0]$ & $[2,0,5,4,1,6,3]$ \\ \hline
 $(2,2)$ & $S_{6}$ & $[1,2,3,4,5,0]$ & $[1,3,0,4,5,2]$ \\ \hline
 $(5,1)$ & ? & ? & ? \\ \hline
 $(6,1)$ & $S_{10}$ & $[1,2,3,4,5,6,7,8,9,0]$ & $[6,4,1,7,3,0,8,9,2,5]$ \\ \hline
 $(3,2)$ & $S_{8}$  &  $[1,2,3,4,5,6,7,0]$  & $[3,5,0,1,6,7,2,4]$  \\ \hline
 $(7,1)$ & ? & ? & ? \\ \hline
 $(8,1)$ & $S_{10}$ & $[1,2,3,4,5,6,7,8,9,0]$ & $[3,0,8,2,6,4,1,9,7,5]$ \\ \hline
 $(4,2)$ & $S_{10}$ & $[1,2,3,4,5,6,7,8,9,0]$ & $[4,2,6,0,1,7,8,3,9,5]$ \\ \hline
 $(9,1)$ & $S_{7}$ & $[1,2,3,4,5,6,0]$ & $[3,0,5,2,6,4,1]$ \\ \hline
 $(3,3)$ & $S_{7}$ & $[1,2,3,4,5,6,0]$ & $[1,2,4,0,5,6,3]$ \\ \hline
 $(10,1)$ & $S_{7}$ & $[1,2,3,4,5,6,0]$ & $[2,0,5,4,1,6,3]$ \\ \hline
 $(5,2)$ & $S_{7}$ & $[1,2,3,4,5,6,0]$ & $[1,3,4,5,0,6,2]$ \\ \hline
 $(11,1)$ & ? & ? & ? \\ \hline
 $(12,1)$ & $S_{10}$ & $[1,2,3,4,5,6,7,8,9,0]$ & $[5,4,3,0,7,8,1,2,9,6]$ \\ \hline
 $(6,2)$ & $S_{6}$ & $[1,2,3,4,5,0]$ & $[1,3,0,4,5,2]$ \\ \hline
 $(4,3)$ & $S_{10}$ & $[1,2,3,4,5,6,7,8,9,0]$ & $[3,0,5,6,1,4,7,8,9,2]$ \\ \hline
 $(13,1)$ & ? & ? & ? \\ \hline
 $(14,1)$ & $S_{10}$ & $[1,2,3,4,5,6,7,8,9,0]$ & $[3,0,8,2,6,4,1,9,7,5]$ \\ \hline
 $(7,2)$ & $S_{10}$ & $[1,2,3,4,5,6,7,8,9,0]$ & $[3,8,6,5,1,4,0,9,7,2]$ \\ \hline
 $(15,1)$ & $S_{7}$ & $[1,2,3,4,5,6,0]$ & $[3,0,5,2,6,4,1]$ \\ \hline
 $(5,3)$ & $S_{7}$ & $[1,2,3,4,5,6,0]$ & $[3,4,5,1,6,0,2]$ \\ \hline
 $(16,1)$ & $S_{7}$ & $[1,2,3,4,5,6,0]$ & $[2,0,5,4,1,6,3]$ \\ \hline
 $(8,2)$ & $S_{7}$ & $[1,2,3,4,5,6,0]$ & $[1,3,4,5,0,6,2]$ \\ \hline
 $(4,4)$ & $S_{7}$ & $[1,2,3,4,5,6,0]$ & $[1,3,4,5,0,6,2]$ \\ \hline
 $(17,1)$ & ? & ? & ? \\ \hline
 $(18,1)$ & $S_{10}$ & $[1,2,3,4,5,6,7,8,9,0]$ & $[6,4,1,7,3,0,8,9,2,5]$ \\ \hline
 $(9,2)$ & $S_{8}$ & $[1,2,3,4,5,6,7,0]$ & $[3,5,0,1,6,7,2,4]$ \\ \hline
 $(6,3)$ & $S_{8}$ & $[1,2,3,4,5,6,7,0]$ & $[5,2,4,0,6,1,7,3]$ \\ \hline
 $(19,1)$ & ? & ? & ? \\ \hline
 $(20,1)$ & $S_{10}$ & $[1,2,3,4,5,6,7,8,9,0]$ & $[3,0,8,2,6,4,1,9,7,5]$ \\ \hline
 $(10,2)$ & $S_{6}$ & $[1,2,3,4,5,0]$ & $[1,3,0,4,5,2]$ \\ \hline
 $(21,1)$ & $S_{7}$ & $[1,2,3,4,5,6,0]$ & $[3,0,5,2,6,4,1]$ \\ \hline
 $(7,3)$ & $S_{7}$ & $[1,2,3,4,5,6,0]$ & $[3,4,5,1,6,0,2]$ \\ \hline
 $(22,1)$ & $S_{7}$ & $[1,2,3,4,5,6,0]$ & $[2,0,5,4,1,6,3]$ \\ \hline
 $(11,2)$ & $S_{7}$ & $[1,2,3,4,5,6,0]$ & $[1,3,4,5,0,6,2]$ \\ \hline
 $(23,1)$ & ? & ? & ? \\ \hline
 $(24,1)$ & ? & ? & ? \\ \hline
 $(12,2)$ & ? & ? & ? \\ \hline
 $(8,3)$ & ? & ? & ? \\ \hline
 $(6,4)$ & ? & ? & ? \\ \hline
 $(25,1)$ & $S_{8}$ & $[1,2,3,4,5,6,7,0]$ & $[2,0,6,5,3,1,7,4]$ \\ \hline
 $(5,5)$ & $S_{8}$ & $[1,2,3,4,5,6,7,0]$ & $[2,3,4,5,6,0,7,1]$ \\ \hline
 $(26,1)$ & $S_{8}$ & $[1,2,3,4,5,6,7,0]$ & $[5,2,4,0,6,1,7,3]$ \\ \hline
 $(13,2)$ & $S_{8}$ & $[1,2,3,4,5,6,7,0]$ & $[3,5,0,1,6,7,2,4]$ \\ \hline
 $(27,1)$ & $S_{7}$ & $[1,2,3,4,5,6,0]$ & $[3,0,5,2,6,4,1]$ \\ \hline
 $(9,3)$ & $S_{7}$ & $[1,2,3,4,5,6,0]$ & $[1,2,4,0,5,6,3]$ \\ \hline
% $(28,1)$ & $S_{7}$ & $[1,2,3,4,5,6,0]$ & $[2,0,5,4,1,6,3]$ \\ \hline
% $(14,2)$ & $S_{6}$ & $[1,2,3,4,5,0]$ & $[1,3,0,4,5,2]$ \\ \hline
% $(7,4)$ & $S_{7}$ & $[1,2,3,4,5,6,0]$ & $[1,3,4,5,0,6,2]$ \\ \hline
\end{tabular}
\end{center}
\end{table}

\end{document}